\documentclass[12pt]{amsart}

\usepackage{graphicx}
\usepackage{amssymb}
\usepackage{amsthm}
\usepackage{amsmath}
\usepackage{stmaryrd}
\usepackage{cite}
\usepackage[mathscr]{eucal}
\usepackage{hyperref}
\usepackage[margin=1.0in]{geometry}
\usepackage{comment}

\usepackage{url}
\usepackage{hyperref}

\title[Vanishing and Estimation results for Hodge numbers]{Vanishing and Estimation results for Hodge numbers}
\author{Peter Petersen and Matthias Wink}
\address{Department of Mathematics, UCLA, 520 Portola Plaza, Los Angeles, CA, 90095}

\email{petersen@math.ucla.edu}
\email{wink@math.ucla.edu}

\begin{document}
\newcommand{\Ext}
{\bigwedge\nolimits}
\newcommand{\Hol} {\operatorname{Hol}}
\newcommand{\diam} {\operatorname{diam}}
\newcommand{\Scal} {\operatorname{Scal}}
\newcommand{\scal} {\operatorname{scal}}
\newcommand{\Ric} {\operatorname{Ric}}
\newcommand{\Hess} {\operatorname{Hess}}
\newcommand{\grad} {\operatorname{grad}}
\newcommand{\Sect} {\operatorname{Sect}}
\newcommand{\Rm} {\operatorname{Rm}}
\newcommand{ \Rmzero } {\mathring{\Rm}}
\newcommand{\Rc} {\operatorname{Rc}}
\newcommand{\Curv} {S_{B}^{2}\left( \mathfrak{so}(n) \right) }
\newcommand{ \tr } {\operatorname{tr}}
\newcommand{ \id } {\operatorname{id}}
\newcommand{ \Riczero } {\mathring{\Ric}}
\newcommand{ \ad } {\operatorname{ad}}
\newcommand{ \Ad } {\operatorname{Ad}}
\newcommand{ \dist } {\operatorname{dist}}
\newcommand{ \rank } {\operatorname{rank}}
\newcommand{\Vol}{\operatorname{Vol}}
\newcommand{\dVol}{\operatorname{dVol}}
\newcommand{ \zitieren }[1]{ \hspace{-3mm} \cite{#1}}
\newcommand{ \pr }{\operatorname{pr}}
\newcommand{\diag}{\operatorname{diag}}
\newcommand{\Lagr}{\mathcal{L}}
\newcommand{\av}{\operatorname{av}}
\newcommand{ \floor }[1]{ \lfloor #1 \rfloor }
\newcommand{ \ceil }[1]{ \lceil #1 \rceil }

\newtheorem{theorem}{Theorem}[section]
\newtheorem{definition}[theorem]{Definition}
\newtheorem{example}[theorem]{Example}
\newtheorem{remark}[theorem]{Remark}
\newtheorem{lemma}[theorem]{Lemma}
\newtheorem{proposition}[theorem]{Proposition}
\newtheorem{corollary}[theorem]{Corollary}
\newtheorem{assumption}[theorem]{Assumption}
\newtheorem{acknowledgment}[theorem]{Acknowledgment}
\newtheorem{DefAndLemma}[theorem]{Definition and lemma}

\newenvironment{remarkroman}{\begin{remark} \normalfont }{\end{remark}}
\newenvironment{exampleroman}{\begin{example} \normalfont }{\end{example}}

\newcommand{\R}{\mathbb{R}}
\newcommand{\N}{\mathbb{N}}
\newcommand{\Z}{\mathbb{Z}}
\newcommand{\Q}{\mathbb{Q}}
\newcommand{\C}{\mathbb{C}}
\newcommand{\F}{\mathbb{F}}
\newcommand{\X}{\mathcal{X}}
\newcommand{\D}{\mathcal{D}}
\newcommand{\Cont}{\mathcal{C}}

\renewcommand{\labelenumi}{(\alph{enumi})}
\newtheorem{maintheorem}{Theorem}[]
\renewcommand*{\themaintheorem}{\Alph{maintheorem}}
\newtheorem*{theorem*}{Theorem}
\newtheorem*{corollary*}{Corollary}
\newtheorem*{remark*}{Remark}
\newtheorem*{example*}{Example}
\newtheorem*{question*}{Question}

\begin{abstract}
We show that compact K\"ahler manifolds have the rational cohomology ring of complex projective space provided a weighted sum of the lowest three eigenvalues of the K\"ahler curvature operator is positive. This follows from a more general vanishing and estimation theorem for the individual Hodge numbers.

We also prove an analogue of Tachibana's theorem for K\"ahler manifolds.
\end{abstract}

\maketitle

\section*{Introduction}

A major topic in geometry is the question how curvature conditions restrict the topology of the manifold. In the case of K\"ahler manifolds, vanishing results for harmonic forms imply restrictions on the Hodge numbers. This principle goes back to Bochner \cite{BochnerVectorFieldsAndRic}, who proved that compact K\"ahler manifolds with positive Ricci curvature cannot admit non-vanishing holomorphic $p$-forms, i.e. $h^{p,0} = 0$ for $1 \leq p \leq n,$ where $n$ denotes the complex dimension of the manifold. In fact, Bochner proved that if the Ricci curvature is $k$-positive, i.e. if the sum of the lowest $k$ eigenvalues of the Ricci tensor is positive, then $h^{p,0}=0$ for $k \leq p \leq n.$ In particular, $h^{n,0}=0$ provided the scalar curvature is positive. Similar results have been obtained by Greene-Wu \cite{GreeneWuCurvatureAndComplexAnalysis} in the non-compact case and by Kobayashi-Wu \cite{KobayashiWuHolomorphicSectionsOfHermitianVB} in the case of compact Hermitian manifolds. 

X. Yang \cite{YangRCpositivity} proved that compact K\"ahler manifolds with positive holomorphic sectional curvature also satisfy $h^{p,0}=0$ for $1 \leq p \leq n$, hence they are projective and moreover rationally connected. This settled one of Yau's problems \cite[Problem 47]{YauProblems}.

Ni-Zheng \cite{NiZhengPositivityAndKodaira} generalized Yang's result by similarly showing that $h^{p,0}=0$ for $k \leq p \leq n$ for compact K\"ahler manifolds with $k$-positive scalar curvature. For $k=1$ this condition reduces to positive holomorphic sectional curvature, whereas for $k=n$ it is positive scalar curvature. In particular, Ni-Zheng show that compact K\"ahler manifolds with $2$-positive scalar curvature cannot admit non-vanishing holomorphic $2$-forms and hence are projective. In previous work, Ni-Zheng \cite{NiZhengComparisonAndVanishingKaehler} similarly proved that K\"ahler manifolds with positive orthogonal Ricci curvature satisfy $h^{2,0}=0$, hence are projective.

The Bochner technique has also been used to control the second de Rham cohomology of compact K\"ahler manifolds, e.g. Bishop-Goldberg \cite{BishopGoldbergCohomKaehler} proved that $b_2(M) =1$ provided $M$ has positive bisectional curvature. In fact, in this case $M$ is biholomorphic to $\mathbb{CP}^n$ according to the solution of the Frankel conjecture due to Mori \cite{MoriProjectiveManifoldsWithAmpleTangentBundles} and Siu-Yau \cite{SiuYauCompactKaehlerPosBiholCurv}. In \cite{ChenSunTianKaehlerRicciSolitons} Chen-Sun-Tian gave an independent proof using the K\"ahler-Ricci flow. Moreover, it follows from the work of Chen \cite{ChenPosOrthBisectCurv} and Gu-Zhang \cite{GuZhangExtensionMokTheorem} that the K\"ahler-Ricci flow evolves metrics with positive orthogonal bisectional curvature into metrics with positive bisectional curvature. In \cite{WilkingALieAlgebriaicApproach} Wilking provided a different proof of this result. As an intermediate step, he used the Bochner technique to show that K\"ahler manifolds with positive orthogonal bisectional curvature satisfy $b_2(M)=1.$

In this paper we offer a different application of the Bochner technique to K\"ahler manifolds. Our methods imply vanishing results for all Hodge numbers $h^{p,q}$ for $1 \leq p,  q \leq n.$

Recall that the curvature operator of the underlying Riemannian manifold $(M,g)$ vanishes on the orthogonal complement of the holonomy algebra $\mathfrak{u}(n) \subset \mathfrak{so}(2n)$. It is therefore natural to study the induced {\em K\"ahler curvature operator} $\mathfrak{R} \colon \mathfrak{u}(n) \to \mathfrak{u}(n)$ with corresponding eigenvalues $\lambda_1 \leq \ldots \leq \lambda_{n^2}$.

Our first main theorem is

\begin{maintheorem}
Let $(M,g)$ be a compact connected K\"ahler manifold of complex dimension $n$. 

If 
\begin{align*}
\lambda_1 + \lambda_2 + \left(1 - \frac{2}{n} \right) \lambda_3 >0,
\end{align*}
then $(M,g)$ has the rational cohomology ring of $\mathbb{CP}^n.$
\label{MainVanishingTheorem}
\end{maintheorem}

Notice that K\"ahler manifolds with $2$-positive K\"ahler curvature operator have positive orthogonal bisectional curvature, and thus are biholomorphic to $\mathbb{CP}^n.$ In particular,  Theorem \ref{MainVanishingTheorem} is known in dimension $n=2.$

Already in dimension $n=2$ similar positivity conditions on the lowest three eigenvalues do not imply that the manifold has positive orthogonal bisectional curvature. In example \ref{EigenvalueConditionIndependentForOrthBisectCurv}, we exhibit for every $\varepsilon > 0$  an algebraic K\"ahler curvature operator $\mathfrak{R} \colon \mathfrak{u}(2) \to \mathfrak{u}(2)$ which does not have positive orthogonal bisectional curvature while its eigenvalues satisfy $\lambda_1 + \lambda_2 < 0$ and $\lambda_1 + \lambda_2 + \varepsilon \lambda_3 >0.$ Moreover, $\mathfrak{R}$ can be chosen to be Einstein. \vspace{2mm}

Theorem \ref{MainVanishingTheorem} follows from a more refined vanishing result for the individual Hodge numbers $h^{p,q}.$ Due to Serre duality, we may assume that $p+q \leq n$ and define
\begin{align*}
C^{p,q} = n+1 - \frac{p^2+q^2}{p+q}. 
\end{align*} 
Notice that $C^{p,p}= n+1-p$ and if $p \geq q$ then $C^{p,q} \geq C^{p+1,q-1}$. We will use the above convention throughout the paper.

\begin{maintheorem}
Let $(M,g)$ be a compact connected K\"ahler manifold of complex dimension $n$.

If 
\begin{align*}
\lambda_1 + \ldots + \lambda_{n+1-p} >0,
\end{align*}
 then $h^{p,p}=1$.

Suppose that $p \neq q.$ If 
\begin{align*}
\lambda_1 + \ldots + \lambda_{\floor{C^{p,q}}} + \left( C^{p,q} - \floor{C^{p,q}} \right) \cdot  \lambda_{\floor{C^{p,q}}+1} >0,
\end{align*}
then $h^{p,q}=0$.

In particular, if $\lambda_1 + \ldots + \lambda_{\floor{C^{p,q}}} > 0,$ then $h^{p,q} = 0.$
\label{VanishingHodgeNumbers}
\end{maintheorem}

In case $p=0$ or $q=0$ Theorem \ref{VanishingHodgeNumbers} follows from Bochner's \cite{BochnerVectorFieldsAndRic} work since K\"ahler manifolds with $n$-positive K\"ahler curvature operators have positive Ricci curvature. Similarly, Theorem \ref{MainVanishingTheorem} follows from Theorem \ref{VanishingHodgeNumbers} and Bochner's observation that K\"ahler manifolds with positive Ricci curvature satisfy $h^{n-1,0}=h^{n,0} = 0$. 

If the K\"ahler curvature operator is merely $3$-positive, then the only forms not controlled by  Theorem \ref{VanishingHodgeNumbers} or Bochner's work are primitive $(n-1,1)$-forms. 

\vspace{2mm}

Many of the previously mentioned results have rigidity analogues. Howard-Smyth-Wu \cite{HowardSmythWuNonnegBiSect} and Wu \cite{WuCompactKaehlerNonnegBiSectII} studied compact K\"ahler manifolds with nonnegative bisectional curvature, and Mok \cite{MokUniformizationKaehler} finally gave a complete classification. Gu \cite{GuNewProofGeneralizedFrankelConj} gave a new proof using Ricci flow methods and Gu-Zhang \cite{GuZhangExtensionMokTheorem} extended the result to nonnegative orthogonal bisectional curvature. 

Due to Bochner's work \cite{BochnerVectorFieldsAndRic}, on a K\"ahler manifold with $k$-nonnegative Ricci curvature every harmonic $(p,0)$-form is parallel for $k \leq p \leq n.$ Similarly, we have 

\begin{maintheorem}
Let $(M,g)$ be a compact K\"ahler manifold of complex dimension $n$. 

If
\begin{align*}
\lambda_1 + \ldots + \lambda_{\floor{C^{p,q}}} + \left( C^{p,q} - \floor{C^{p,q}} \right) \cdot \lambda_{\floor{C^{p,q}}+1} \geq 0,
\end{align*}
then every harmonic $(p,q)$-form is parallel. 

In particular, if $\lambda_1 + \ldots + \lambda_{\floor{C^{p,q}}} \geq  0,$ then every harmonic $(p,q)$-form is parallel and specifically if $\lambda_1 + \ldots + \lambda_{n+1-p} \geq 0,$ then every $(p,p)$-form is parallel.
\label{MainRigidityTheorem}
\end{maintheorem}

Combined with the observation that harmonic $(n,0)$-forms and $(n-1,0)$-forms are parallel if the Ricci curvature is positive, Theorem \ref{MainRigidityTheorem} implies the following global result.

\begin{corollary*}
Let $(M,g)$ be an $n$-dimensional K\"ahler manifold. If 
\begin{align*}
\lambda_1 + \lambda_2 + \left( 1 - \frac{2}{n} \right) \lambda_3 \geq 0,
\end{align*}
then every harmonic form is parallel.
\end{corollary*}

Recall that the Riemannian curvature operator of a K\"ahler manifold has a kernel of dimension at least $n(n-1)$. Therefore the results in \cite{PetersenWinkNewCurvatureConditionsBochner} reduce to Gallot and D. Meyer's  \cite{GallotMeyerCurvOperatorAndForms} rigidity theorem for manifolds with nonnegative curvature operator, when the Riemannian manifold is K\"ahler. \vspace{2mm}

Due to the work of P. Li \cite{LiSobolevConstant} and Gallot \cite{GallotSobolevEstimates}, the Bochner technique also implies estimation results provided a lower bound on the Ricci curvature and an upper bound on the diameter are assumed. In the situation of Theorem \ref{MainEstimationTheorem} this follows from the fact that the Ricci curvature is bounded from below by the sum of the lowest $n$ eigenvalues of the K\"ahler curvature operator.

\begin{maintheorem} 
Let $\kappa \leq 0$ and $D>0$ and suppose that $(M,g)$ is a compact connected $n$-dimensional K\"ahler manifold with diameter at most $D.$ 

If 
\begin{align*}
\lambda_1 + \ldots + \lambda_{\floor{C^{p,q}}} + \left( C^{p,q} - \floor{C^{p,q}} \right) \cdot \lambda_{\floor{C^{p,q}}+1} \geq \kappa (\floor{C^{p,q}}+1),
\end{align*}
then
\begin{align*}
h^{p,q}(M) \leq {n \choose p} {n \choose q} \exp \left( C(n, \kappa D^2) \cdot \sqrt{- \kappa D^2 \cdot \left(n+2-|p-q| \right)(p+q)} \right).
\end{align*}

In particular, there is $\varepsilon(n) > 0$ such that $\kappa D^2 \geq - \varepsilon(n)$ implies $h^{p,q} \leq {n \choose p} {n \choose q}$.

If 
\begin{align*}
\lambda_1 + \lambda_2 + \left( 1 - \frac{2}{n} \right) \lambda_3 \geq \kappa,
\end{align*}
then the total Betti number is bounded by
\begin{align*}
\sum_{p+q=0}^n h^{p,q} \leq 2^{2n} \exp  \left( C(n, \kappa D^2) \cdot \sqrt{- \kappa D^2} \right).
\end{align*}
\label{MainEstimationTheorem}
\end{maintheorem}

As in Theorems \ref{VanishingHodgeNumbers} and \ref{MainRigidityTheorem}, the conclusion of Theorem \ref{MainEstimationTheorem} also holds if $\lambda_1 + \ldots + \lambda_{\floor{C^{p,q}}} \geq \kappa \floor{C^{p,q}}$ and thus specifically for $h^{p,p}$ if $\lambda_1 + \ldots + \lambda_{n+1-p} \geq \kappa(n+1-p).$ \vspace{2mm}

For a Riemannian manifold, a famous theorem of Tachibana \cite{TachibanaPosCurvOperator} asserts that any Einstein manifold with nonnegative curvature operator is locally symmetric. Moreover, if the curvature operator is positive, then the manifold has constant sectional curvature. Brendle \cite{BrendleEinsteinNIC} generalized this to Einstein metrics with nonnegative, respectively positive, isotropic curvature. In real dimension four this was observed earlier by Micallef and Wang \cite{MicallefWangNIC}. 

Notice that only the rigidity part of these theorems actually applies to K\"ahler manifolds. Tachibana-type results specifically for K\"ahler manifolds follow from the classification results for K\"ahler manifolds of nonnegative, respectively positive, bisectional and orthogonal bisectional curvature due to Mori \cite{MoriProjectiveManifoldsWithAmpleTangentBundles}, Siu-Yau \cite{SiuYauCompactKaehlerPosBiholCurv}, Mok \cite{MokUniformizationKaehler} and Chen \cite{ChenPosOrthBisectCurv}, Gu-Zhang \cite{GuZhangExtensionMokTheorem}.

We have the following analogue of Tachibana's theorem for K\"ahler manifolds.

\begin{maintheorem}
Suppose that $(M,g)$ is a compact connected K\"ahler-Einstein manifold of complex dimension $n \geq 4.$ 

If 
\begin{align*}
\lambda_{1} + \ldots + \lambda_{\floor{\frac{n+1}{2}}} + \frac{1+(-1)^n}{4} \cdot \lambda_{\floor{\frac{n+1}{2}}+1} \geq 0,
\end{align*}
then the curvature tensor is parallel. 

If the inequality is strict, then $(M,g)$ has constant holomorphic sectional curvature.
\label{KaehlerTachibana}
\end{maintheorem}

The Assumptions in Theorem \ref{KaehlerTachibana} are satisfied in particular when $\lambda_{1} + \ldots + \lambda_{\floor{\frac{n+1}{2}}} \geq 0$ or $\lambda_{1} + \ldots + \lambda_{\floor{\frac{n+1}{2}}} > 0,$ respectively.

In \cite{PetersenWinkNewCurvatureConditionsBochner} we show that any Einstein manifold of real dimension $m$ with $\floor{\frac{m-1}{2}}$-nonnegative Riemannian curvature operator is locally symmetric. However, any K\"ahler manifold satisfying this condition in fact has  nonnegative curvature operator and thus the result reduces to Tachibana's \cite{TachibanaPosCurvOperator} original theorem on manifolds with nonnegative curvature operator. \vspace{2mm}

The proofs of the main theorems rely on the Bochner technique. If $(M,g)$ is a Riemannian manifold, the associated Lichnerowicz Laplacian on tensors is 
\begin{align*}
\Delta_L T = \nabla^{*} \nabla T + c \Ric(T)
\end{align*}
where $c>0$ is a constant. For $1$-forms $\varphi$, $\Ric(\varphi)$ is determined by the Ricci curvature but otherwise $\Ric(T)$ depends on the entire Riemannian curvature tensor. 

Our new approach explains how the action of the holonomy algebra $\mathfrak{g}$ on tensors simplifies the curvature term of the Lichnerowicz Laplacian. Specifically we show that every complex valued $(0,r)$-tensor $T$ satisfies
\begin{align*}
g( \Ric(T), \overline{T} ) = \sum_{\Xi_{\alpha} \in \mathfrak{g}} \lambda_{\alpha} | \Xi_{\alpha} T |^2
\end{align*}
where $\lbrace \Xi_{\alpha} \rbrace$ is an orthonormal basis for the restricted curvature operator $\mathfrak{R}_{|\mathfrak{g}} \colon \mathfrak{g} \to \mathfrak{g}$ and $\lbrace \lambda_{\alpha} \rbrace$ are the corresponding eigenvalues. This generalizes Poor's \cite{PoorHolonomyProofPosCurvOperatorThm} idea of using the derivative of the regular representation to study the curvature term on $p$-forms. 

The key insight in gaining control on the curvature term is that if $E$ is a holonomy irreducible tensor bundle, then there are constants $c(E) \leq C(E)$ such that $| \Xi_{\alpha} T |^2 \leq c(E) \cdot |T|^2$ while $\sum_{\Xi_{\alpha} \in \mathfrak{g}} | \Xi_{\alpha} T |^2 = C(E) \cdot |T|^2$. Lemma \ref{WeightedGeneralBochnerLemma} then provides a method to estimate $g( \Ric(T), \overline{T} )$ based on a lower bound on a weighted sum of the eigenvalues of the curvature operator $\mathfrak{R}_{|\mathfrak{g}} \colon \mathfrak{g} \to \mathfrak{g}$.

The proofs of Theorems \ref{MainVanishingTheorem} - \ref{MainEstimationTheorem} are an application of this principle to $(p,q)$-forms on K\"ahler manifolds. In particular, they use the decomposition of the space of $(p,q)$-forms into $U(n)$-irreducible modules. Theorem \ref{KaehlerTachibana} is a similar application of our technique to the space of K\"ahler curvature operators. \vspace{2mm}

Section \ref{PreliminariesSection} introduces Lichnerowicz Laplacians and the relevant background material. Section \ref{SectionIrrReps} discusses the decomposition of $(p,q)$-forms into $U(n)$-irreducible modules, originally due to Chern \cite{ChernGeneralizationKaehlerGeometry}. In section \ref{SectionEstimationCurvatureTerm} we study the Lichnerowicz Laplacian on $(p,q)$-forms. In particular, lemma \ref{EstimateLieAlgebraActionpqForms} and proposition \ref{LichnerowiczEstimatePQForms} establish the required estimates to apply lemma \ref{WeightedGeneralBochnerLemma} to the $U(n)$-irreducible modules of the space of $(p,q)$-forms. The proofs of Theorems \ref{MainVanishingTheorem} - \ref{MainEstimationTheorem} are given in section \ref{SectionProofOfMainTheorems} and Theorem \ref{KaehlerTachibana} is proven in section \ref{SectionKaehlerTachibana}. \vspace{2mm}

\textit{Acknowledgements.} We would like to thank Greg Kallo for many conversations.

\section{Preliminaries}
\label{PreliminariesSection}

\subsection{Tensors}
\label{SubsectionOnTensorsAndHat}

Let $(V,g)$ be an $m$-dimensional Euclidean vector space and let $\operatorname{Sym}^2(V) \subset \bigotimes^2 V^{*}$ denote the space of symmetric $(0,2)$-tensors on $V.$ 

The metric $g$ induces a metric on $\bigotimes^{r} V^{*}$ and $\Ext^r V$. In particular, if $\lbrace e_i \rbrace_{i=1, \ldots, m}$ is an orthonormal basis for $V$, then $\left\lbrace  e_{i_1} \wedge \ldots \wedge e_{i_r} \right\rbrace _{1 \leq i_1 < \ldots < i_r \leq m}$ is an orthonormal basis for $\Ext^r V.$ This also induces an inner product on $\mathfrak{so}(V)$ via its identification with $\Ext^2 V.$ 

Let $V_{\C} = V \otimes_{\R} \C.$ For a complex valued, $\R$-multilinear tensor $T$ on $V$, i.e. $T \in \bigotimes^r V_{\C}^{*},$ and $L \in \mathfrak{so}(V)$ set
\begin{align*}
(LT)(X_1, \ldots, X_r) = - \sum_{i=1}^r T(X_1, \ldots, LX_i, \ldots, X_r).
\end{align*}

If $\mathfrak{g} \subset \mathfrak{so}(V)$ is a Lie subalgebra, define $T^{\mathfrak{g}} \in \left( \bigotimes^r V_{\C}^{*} \right) \otimes_{\R} \mathfrak{g} $ by 
\begin{equation*}
g( L, T^{\mathfrak{g}}(X_1, \ldots, X_r)) = (LT)(X_1, \ldots, X_r)
\end{equation*}
for all $L \in \mathfrak{g} \subset \mathfrak{so}(V) = \Ext^2V$.
Furthermore, if $\mathfrak{R} \colon \mathfrak{g} \to \mathfrak{g}$ is a self-adjoint operator with orthonormal eigenbasis $\lbrace \Xi_{\alpha} \rbrace$ and corresponding eigenvalues $\lbrace \lambda_{\alpha} \rbrace$, then 
\begin{align*}
\mathfrak{R}(T^{\mathfrak{g}}) = \mathfrak{R} \circ T^{\mathfrak{g}} = \sum_{\alpha} \mathfrak{R}( \Xi_{\alpha} ) \otimes \Xi_{\alpha} T
\end{align*}
and as a consequence we obtain
\begin{align*}
g( \mathfrak{R}(T^{\mathfrak{g}}), \overline{T}{^{\mathfrak{g}}}) 
= \sum_{\alpha} \lambda_{\alpha} | \Xi_{\alpha} T |^2. 
\end{align*} 
In particular, notice that 
\begin{align*}
|T^{\mathfrak{g}}|^2 = \sum_{\alpha} | \Xi_{\alpha} T |^2.
\end{align*}

In case $\mathfrak{g}=\mathfrak{u}(n),$ we will write $T^{\mathfrak{u}}$ to simplify notation.

\begin{remarkroman}
Let $(M,g)$ be a Riemannian manifold and let $\mathfrak{R} \colon \Ext^2 TM \to \Ext^2 TM$ denote the curvature operator. If $\mathfrak{g} \subset \mathfrak{so}(2n)$ denotes the holonomy algebra, then $\mathfrak{R}_{|\mathfrak{g}} \colon \mathfrak{g} \to \mathfrak{g}$ and $\mathfrak{R}_{|\mathfrak{g}^{\perp}} = 0.$
\end{remarkroman}

\begin{example} \normalfont
Let $\mathfrak{g} \subset \mathfrak{so}(V) = \Ext^2 V$ be a Lie subalgebra. For a self-adjoint operator $\mathfrak{R} \colon \mathfrak{g} \to \mathfrak{g}$ let $R \in \operatorname{Sym}^2( \mathfrak{g}) \subset \operatorname{Sym}^2 \left( \Ext^2 V \right)$ denote the corresponding bilinear form. $R$ is an algebraic curvature tensor if it satisfies the first Bianchi identity. In this case we write $R \in \operatorname{Sym}_B^2 \left( \mathfrak{g} \right).$

The proof of \cite[Proposition 1.5]{PetersenWinkNewCurvatureConditionsBochner} shows that if $\lbrace \Xi_{\alpha}\rbrace$ is an orthonormal eigenbasis for $\mathfrak{R}$ and $L \in \mathfrak{g}$, then
\begin{align*}
| L R |^2 = 2 \sum_{\alpha < \beta} \left( \lambda_{\alpha} - \lambda_{\beta} \right)^2 g( L \Xi_{\alpha}, \Xi_{\beta} )^2.
\end{align*}
It follows that 
\begin{align*}
| R{^{\mathfrak{g}}} |^2 = 2  \sum_{\gamma} \sum_{\alpha < \beta} \left( \lambda_{\alpha} - \lambda_{\beta} \right)^2 g( ( \Xi_{\gamma} ) \Xi_{\alpha}, \Xi_{\beta} )^2.
\end{align*}

Notice that $\mathfrak{so}(V)$ induces a Lie bracket $[ \cdot, \cdot ]$ on $\Ext^2 V.$ For $2$-forms $\Xi_{\alpha}, \Xi_{\beta}$ we have 
\begin{align*}
(\Xi_{\alpha}) \Xi_{\beta} = [ \Xi_{\alpha}, \Xi_{\beta}].
\end{align*}

In particular, the coefficients $g( ( \Xi_{\gamma} ) \Xi_{\alpha}, \Xi_{\beta} )$ are the structure constants and $g( ( \Xi_{\gamma} ) \Xi_{\alpha}, \Xi_{\beta} )^2$ is fully symmetric in $\Xi_{\alpha},$ $\Xi_{\beta},$ $\Xi_{\gamma}.$
\label{HatsOfCurvatureTensors}
\end{example}

\subsection{The Lie algebra $\mathfrak{u}(V)$}
\label{SubsectionLieAlgebraUn}
Suppose $(V,g)$ is a $2n$-dimensional Euclidean vector space with compatible almost complex structure $J \colon V \to V.$ It follows that
\begin{align*}
\mathfrak{u}(V) = \left\lbrace L \in \mathfrak{gl}(V) \ \vert \ L  \circ J = J \circ L, \ g(L \cdot, \cdot ) + g( \cdot, L \cdot ) = 0 \right\rbrace.
\end{align*}

Let $e_1, \ldots, e_n, f_1=Je_1, \ldots, f_n= Je_n$ be an orthonormal basis for $V.$ Under the identification of $\Ext^2 V$ with $\mathfrak{so}(V),$ an orthonormal basis for $\mathfrak{u}(V)$ is given by
\begin{align*}
R_{ij} & = \frac{1}{\sqrt{2}} \left( e_i \wedge e_j + f_i \wedge f_j \right) \ \text{for} \ 1 \leq i < j \leq n, \\
I_{ij} & = \frac{1}{\sqrt{2}} \left( e_i \wedge f_j + e_j \wedge f_i \right) \ \text{for} \ 1 \leq i < j \leq n,  \\
I_{ii} & =  e_i \wedge f_i \ \text{for} \ 1 \leq i \leq n.
\end{align*}
Note that $R_{ij}=-R_{ji}$ and $I_{ij}=I_{ji}.$

Moreover, an orthonormal basis for $\mathfrak{u}(V)^{\perp} \subset \mathfrak{so}(V)$ is given by
\begin{align*}
\left( R_{ij} \right)^{\perp} & = \frac{1}{\sqrt{2}} \left( e_i \wedge e_j - f_i \wedge f_j \right) \ \text{for} \ 1 \leq i < j \leq n, \\
\left( I_{ij} \right)^{\perp} & = \frac{1}{\sqrt{2}} \left( e_i \wedge f_j - e_j \wedge f_i \right) \ \text{for} \ 1 \leq i < j \leq n.
\end{align*}

The space of complex valued, $\R$-linear $1$-forms on $V$ decomposes into $\C$-linear and conjugate linear forms,
\begin{align*}
\Ext^1 V^{*}_{\C} = \Ext^{1,0} V^{*} \oplus \Ext^{0,1} V^{*}.
\end{align*}
Thus if $dx^{1}, \ldots, dx^{n},$ $dy^{1}, \ldots, dy^{n}$ denotes the dual basis, then 
\begin{align*}
dz^{i} & = dx^{i} + \sqrt{-1} dy^{i} \in \Ext^{1,0} V^{*}, \\
d \bar{z}^{i} & = dx^{i} - \sqrt{-1} dy^{i} \in \Ext^{0,1} V^{*}.
\end{align*} 
Furthermore, the {\em K\"ahler form} $\omega( \cdot , \cdot ) =g(J \cdot , \cdot)$ is given by
\begin{align*}
\omega = \frac{\sqrt{-1}}{2} \sum_{i=1}^n dz^i \wedge d \bar{z}^i.
\end{align*}

\begin{proposition}
Let $i \neq j$. The following hold:
\begin{align*}
(R_{ij}) dz^{i} & = - \frac{1}{\sqrt{2}} dz^j, \ 
(R_{ij}) d \bar{z}^{i} = - \frac{1}{\sqrt{2}} d \bar{z}^j, \\
(I_{ij}) dz^{i} & =  \frac{\sqrt{-1}}{\sqrt{2}} dz^{j}, \hspace{2.8mm} (I_{ij}) d \bar{z}^{i} = - \frac{\sqrt{-1}}{\sqrt{2}} d \bar{z}^j,  \\
(I_{ii}) dz^{i} & =  \sqrt{-1} dz^i, \hspace{4.4mm}
(I_{ii}) d \bar{z}^{i}  = - \sqrt{-1} d \bar{z}^i.
\end{align*}
\label{ActionOfLieAlgebraUnOnPQForms}
\end{proposition}
\begin{proof}
This is a straightforward calculation. Notice that e.g. $(e_i \wedge f_j) dx^{i} = - dy^{j}$.
\end{proof}

\begin{remarkroman}
The K\"ahler form is in the kernel of the Lie algebra action of $\mathfrak{u}(n).$ That is, for $i \neq j$ we have $R_{ij} \omega = I_{ij} \omega = I_{ii} \omega = 0.$ Furthermore, in corollary \ref{HatZeroOnlyOnMultiplesOfKaehlerForm} we show that $\varphi \in \Ext^{k} V^{*}_{\C}$ satisfies $|\varphi ^{\mathfrak{u}} |^2 = 0$ if and only if $\varphi=0$ or $k$ is even and $\varphi$ is a multiple of $\omega^{k/2}$.
\label{KaehlerFormInKernel}
\end{remarkroman}

\subsection{Lichnerowicz Laplacians and holonomy}
\label{SectionHolonomy}

Let $(M,g)$ be a Riemannian manifold. For $c>0$ the Lichnerowicz Laplacian on $(0,r)$-tensors is given by 
\begin{align*}
\Delta_L T = \nabla^{*} \nabla T + c \Ric( T )
\end{align*}
where 
\begin{align*}
\Ric( T )(X_1, \ldots, X_r) = \sum_{i=1}^r \sum_{j=1}^m (R(X_i,e_j)T) (X_1, \ldots, e_j, \ldots, X_r)
\end{align*}
for an orthonormal frame $e_1, \ldots, e_{m}$ of the tangent bundle $TM.$ 

Let $(M,g)$ be connected with holonomy group $\Hol(g).$ The holonomy representation induces a representation on the tensor bundle $\bigotimes^k T^{*}_{\C}M.$ Suppose that $E$ is an invariant subbundle. If $T \in \Gamma(E)$, then $\nabla^{*} \nabla T \in \Gamma(E)$. Furthermore, since the Riemannian curvature tensor takes values in the holonomy algebra, it also follows that $\Ric( T ) \in \Gamma(E).$ Thus the Lichnerowicz Laplacian preserves subbundles $E$ which are invariant under the holonomy representation, $\Delta_L \colon \Gamma(E) \to \Gamma(E).$ Moreover, $E$ decomposes into a direct sum of $\Hol(g)$-irreducible subbundles.

A tensor $T$ is {\em harmonic} if $\Delta_L T =0$ and in this case we have the Bochner formula
\begin{align*}
\Delta \frac{1}{2} |T|^2 = | \nabla T |^2 + g( \Ric(T), \overline{T}).
\end{align*}
The Bochner technique is based on the fact that if $g( \Ric(T), \overline{T}) \geq 0$ and $|T|$ has a maximum, then $T$ is parallel.

\begin{exampleroman}
The Hodge Laplacian is a Lichnerowicz Laplacian for $c=1$. It follows that the decomposition of $\Ext^k T^{*}_{\C}M$ into $\Hol(g)$-irreducible modules induces a decomposition of harmonic forms and the de Rham cohomology groups.
\label{HodgeLaplacianAsLichnerowicz}
\end{exampleroman}

The following proposition is an immediate consequence of \cite[Lemmas 9.3.3 and 9.4.3]{PetersenRiemGeom}.

\begin{proposition} 
\label{CurvatureTermLichnerowiczLaplacian}
Let $\mathfrak{R} \colon \Ext^2 TM \to \Ext^2 TM$ denote the curvature operator of $(M,g)$. If $\mathfrak{g} \subset \mathfrak{so}(m)$ denotes the holonomy algebra, then $\mathfrak{R}_{|\mathfrak{g}} \colon \mathfrak{g} \to \mathfrak{g}$, $\mathfrak{R}_{|\mathfrak{g}^{\perp}} = 0$ and
\begin{equation*}
g( \Ric(T), \overline{T}) = g( \mathfrak{R}_{|\mathfrak{g}}( T^{\mathfrak{g}}), \overline{T}{^{\mathfrak{g}}} )
\end{equation*}
for every $T \in \bigotimes^r T^{*}_{\C} M$.
\end{proposition}

\begin{remarkroman}
Recall that an irreducible Riemannian manifold $(M,g)$ is Einstein unless its holonomy group is $SO(n)$ or $U(n)$. The reader is referred to \cite{PetersenWinkNewCurvatureConditionsBochner} for the case $\Hol(g)=SO(n)$. In this paper we will restrict ourselves to $\Hol(g)=U(n)$. Recall that $\Hol(g) \subset SU(n)$ if and only if there exists a parallel holomorphic volume form and in this case $(M,g)$ is Ricci flat.
\end{remarkroman}

Suppose that $\Hol(g)$ is contained in $U(n) \subset SO(2n),$ i.e. $(M,g)$ is a K\"ahler manifold of complex dimension $n.$ 
The induced curvature operator $\mathfrak{R} = \mathfrak{R}_{| \mathfrak{u}(n)} \colon \mathfrak{u}(n) \to \mathfrak{u}(n)$ is the {\em K\"ahler curvature operator.}

The following lemma is the fundamental tool for controlling the curvature term of the Lichnerowicz Laplacian.

\begin{lemma}
Let $(V,g)$ be a Euclidean vector space, $\mathfrak{g} \subset \mathfrak{so}(V)$ a Lie subalgebra and let $\mathfrak{R} \colon \mathfrak{g} \to \mathfrak{g}$ be self-adjoint with eigenvalues $\lambda_1 \leq \ldots \leq \lambda_{\dim \mathfrak{g}}$. Let $T \in \bigotimes^r V^{*}_{\C}$ and suppose that there is $C \geq 1$ such that
\begin{align*}
| L T |^2 \leq \frac{1}{C} | T^{\mathfrak{g}} |^2 |L|^2 
\end{align*}
for all $L \in \mathfrak{g}$. Let $1 \leq l \leq \floor{C}$ be an integer and let $\kappa \leq 0$. 

If $\lambda_1 + \ldots + \lambda_l + \left( C - l \right) \lambda_{l+1} \geq \kappa (l+1),$ then 
$g( \mathfrak{R}( T^{\mathfrak{g}}), \overline{T}{^{\mathfrak{g}}} ) \geq \frac{\kappa (l+1)}{C} | T^{\mathfrak{g}} |^2$.

If $\lambda_1 + \ldots + \lambda_l + \left( C - l \right) \lambda_{l+1} > 0,$ then 
$g( \mathfrak{R}( T^{\mathfrak{g}}), \overline{T}{^{\mathfrak{g}}} )>0$ unless $T^{\mathfrak{g}} = 0.$
\label{WeightedGeneralBochnerLemma}
\end{lemma}
\begin{proof}
Suppose that $\lbrace \Xi_{\alpha} \rbrace$ is an orthonormal eigenbasis for $\mathfrak{R}.$ It follows as in \cite[Proof of Lemma 2.1]{PetersenWinkNewCurvatureConditionsBochner} that
\begin{align*}
g( \mathfrak{R}( T^{\mathfrak{g}}), \overline{T}{^{\mathfrak{g}}} ) & \geq \lambda_{l+1} \left( 1 - \frac{l}{C} \right) | T^{\mathfrak{g}}|^2 + \frac{| T^{\mathfrak{g}} |^2}{C} \sum_{\alpha}^{l} \lambda_{\alpha} 
= \frac{| T^{\mathfrak{g}} |^2}{C} \left( \sum_{\alpha}^{l} \lambda_{\alpha}  + ( C - l) \lambda_{l + 1} \right),
\end{align*}
which implies the claim.
\end{proof}

Clearly, choosing $l=\floor{C}$ provides the weakest (curvature) assumption. Note that this condition is in particular satisfied if $\lambda_1 + \ldots + \lambda_{\floor{C}} \geq \kappa \floor{C}$ or $\lambda_1 + \ldots + \lambda_{\floor{C}} >0$.

If $C$ is an integer, then the same proof yields that if $\lambda_1 + \ldots + \lambda_C \geq \kappa C$, then $g( \mathfrak{R}( T^{\mathfrak{g}}), \overline{T}{^{\mathfrak{g}}} ) \geq \kappa | T^{\mathfrak{g}} |^2.$

\section{$U(n)$-irreducible decomposition of $(p,q)$-forms}
\label{SectionIrrReps}

In this section we provide a description of the decomposition of $(p,q)$-forms into irreducible $U(n)$-modules. This is orginially due to Chern \cite{ChernGeneralizationKaehlerGeometry}, see also Fujiki \cite{FujikiCohomKaehlerSymplectic}. For completeness, we provide an elementary proof using characters. 

Let $V= \C^n$ and consider the natural $U(n)$-action on $V.$ Denote by
\begin{align*}
\Ext^{p,0} V^{*} = \Ext^p V^{*} = \operatorname{span}_{\C} \lbrace dz^{i_1} \wedge \ldots \wedge dz^{i_p} \ \vert \ 1 \leq i_1 < \ldots < i_p \leq n \rbrace
\end{align*}
the space of complex linear $p$-forms, by 
\begin{align*}
\Ext^{0,q} V^{*} = \Ext^q \overline{V^*} = \operatorname{span}_{\C} \lbrace d \bar{z}^{j_1} \wedge \ldots \wedge d \bar{z}^{j_q} \ \vert \ 1 \leq j_1 < \ldots < j_q \leq n \rbrace
\end{align*} the space of conjugate linear $q$-forms, and by
\begin{align*}
\Ext^{p,q} V^* = \Ext^{p,0} V^* \otimes_{\C} \Ext^{0,q} V^*
\end{align*}
the space of $(p,q)$-forms. 

For $0 \leq k \leq \min \lbrace p, q \rbrace$ set 
\begin{align*}
V^{p,q}_k = \Ext^{p-k,0} V^{*} \otimes_{\C} \operatorname{span}_{\C} \lbrace \omega^k \rbrace \otimes_{\C} \Ext^{0,q-k} V^{*}.
\end{align*}
Hence for $k \leq q \leq p$ we have the flag 
\begin{align*}
V^{p,q}_q   \subseteq \ldots \subseteq  V^{p,q}_{k} \subseteq \ldots \subseteq  V^{p,q}_{1} \subseteq  V^{p,q}_{0} = \Ext^{p,q} V^{*}.
\end{align*}

\begin{theorem}
The representations of $U(n)$ on 
\begin{align*}
\Ext^{p,q}_k V^{*} = V^{p,q}_k \cap \left( V^{p,q}_{k+1} \right)^{\perp}
\end{align*}
are irreducible and
\begin{align*}
\Ext^{p,q} V^{*} = \bigoplus_{k=0}^{\min \lbrace p,q \rbrace} \Ext^{p,q}_k V^{*}
\end{align*}
is an orthogonal decomposition.
\label{IrreducibleUnDecomposition}
\end{theorem}

\begin{remark} \normalfont 
Let $\mathcal{L} \colon \varphi \mapsto \omega \wedge \varphi$ be  the Lefschetz map and let $\Lambda$ denote its dual. A $(p,q)$-form $\varphi$ is {\em primitive} if $\Lambda \varphi =0.$ It follows that $\Ext^{p,q}_0 V^{*}$ is the space of primitive $(p,q)$-forms and $\Ext^{p,q}_k V^{*} = \mathcal{L}^k \Ext^{p-k,q-k}_0 V^{*}$.
\end{remark}

The above decomposition of $\Ext^{p,q} V^{*}$ into $U(n)$-irreducible modules is due to Chern \cite{ChernGeneralizationKaehlerGeometry}.
However, we have not been able to access or obtain a copy of Chern's original proof.

For completness, in rest of this section we show that the character of $\Ext^{p,q}_k$ is the character of an irreducible $U(n)$-representation. Our proof is elementary and only uses Laplace's expansion of a determinant along two rows and Weyl's \cite[Chapter VII, Sections 4-5]{WeylClassicalGroups} classification of irreducible $U(n)$- representations. In particular, our method is different from Fujiki's approach in \cite{FujikiCohomKaehlerSymplectic}. \vspace{2mm}

Recall that the maximal torus $T^n \subset U(n)$ is 
\begin{align*}
T^n = \lbrace \diag(\varepsilon_1, \ldots, \varepsilon_n) \ \vert \ |\varepsilon_i|=1 \rbrace
\end{align*}
and for $U \in T^n$ we have
\begin{align*}
U dz^{i} = \varepsilon_i dz^i, \ U d \bar{z}^{j} = \bar{\varepsilon}_j d \bar{z}^j.
\end{align*}

More generally, let 
\begin{align*}
dz^{I_p} = dz^{i_1} \wedge \ldots \wedge dz^{i_p}, \ d \bar{z}^{J_q} = d \bar{z}^{j_1} \wedge \ldots \wedge d \bar{z}^{j_q}
\end{align*}
where $1 \leq i_1 < \ldots < i_p \leq n$ and $1 \leq j_1 < \ldots < j_q \leq n$. Similarly, define
\begin{align*}
\varepsilon_{I_p}=\varepsilon_{i_1} \cdot \ldots \cdot \varepsilon_{i_p}, \ \bar{\varepsilon}_{J_q}=\bar{\varepsilon}_{j_1} \cdot \ldots \cdot \bar{\varepsilon}_{j_q}.
\end{align*}
It follows that $dz^{I_p} \wedge d \bar{z}^{J_q}$ is an eigenvector of the induced action of the maximal torus with eigenvalue $\varepsilon_{I_p} \bar{\varepsilon}_{J_q}.$ This immediately implies that the character of $\Ext^{p,q} V^{*}$ is 
\begin{align*}
\chi^{p,q}=\sum_{I_p, J_q} \varepsilon_{I_p} \bar{\varepsilon}_{J_q}.
\end{align*}

\begin{remarkroman}
(a) We have the explicit formula
\begin{align*}
\chi^{p,q} = \sum_{k=0}^{\min \lbrace p, q \rbrace} { n-(p+q-2k) \choose k} \sum_{I_{p-k} \cap J_{q-k} = \emptyset} \varepsilon_{I_{p-k}} \bar{\varepsilon}_{J_{q-k}}.
\end{align*}
(b) Note that $V^{p,q}_k$ and $\Ext^{p-k,q-k} V^{*}$ are isomorphic $U(n)$-representations since they both have character $\chi^{p-k,q-k}$.
\label{CommentOnIsomorphicReps}
\end{remarkroman}

Following Weyl's notation in \cite[Chapter VII, Sections 4-5]{WeylClassicalGroups}, for $\varepsilon = ( \varepsilon_1, \ldots, \varepsilon_n )$ define the alternant
\begin{align*}
| \varepsilon^{l_1}, \varepsilon^{l_2}, \ldots, \varepsilon^{l_n} | = \begin{vmatrix}
\varepsilon_1^{l_1} & \varepsilon_1^{l_1-1} & \ldots &  \varepsilon_1^{l_1-n} \\
\varepsilon_2^{l_2} & \varepsilon_2^{l_2-1} & \ldots &  \varepsilon_2^{l_2-n} \\
\vdots & \vdots &  & \vdots \\
\varepsilon_n^{l_n} & \varepsilon_n^{l_n-1} & \ldots &  \varepsilon_n^{l_n-n}
\end{vmatrix}.
\end{align*}

Notice that in particular
\begin{align*}
\Delta = | \varepsilon^{n-1}, \varepsilon^{n-2}, \ldots, \varepsilon,  1 | = \prod_{i<j} ( \varepsilon_i - \varepsilon_j )
\end{align*}
is a Vandermonde determinant. Theorem \ref{WeylCharacterizationUnIrrReps} below is Weyl's classification of irreducible $U(n)$-representations in \cite[Chapter VII, Theorems 7.5.B and 7.5.C]{WeylClassicalGroups}.

\begin{theorem}
Let $f_1 \geq f_2 \geq \ldots \geq f_n$ be integers. Every representation of the unitary group $U(n)$ with character
\begin{align*}
\chi_{f_1, \ldots, f_n} = \frac{1}{\Delta} \cdot |\varepsilon^{f_1 + n-1}, \varepsilon^{f_2 + n-2}, \ldots, \varepsilon^{f_n}|
\end{align*}
is irreducible. 

Conversely, every irreducible representation of $U(n)$ has the character $\chi_{f_1, \ldots, f_n}$ for some integers $f_1 \geq f_2 \geq \ldots \geq f_n$.
\label{WeylCharacterizationUnIrrReps}
\end{theorem}

Theorem \ref{IrreducibleUnDecomposition} is now an immediate consequence of the following observation. 

\begin{lemma}
The character of $\Ext^{p,q}_k V^{*}$ is given by 
\begin{align*}
\chi^{p,q}_k & = \chi^{p-k,q-k} - \chi^{p-k-1,q-k-1} = \chi_{f_1+n-1, f_2+n-2, \ldots, f_n}
\end{align*}
for $f_1= \ldots = f_{p-k} =1,$ $f_{p-k+1}= \ldots = f_{n-(q-k)}=0$ and $f_{n-(q-k)+1} = \ldots = f_n =-1.$
\label{CharacterOfExtpqk}
\end{lemma}
\begin{proof}
The definition of $\Ext^{p,q}_k V^{*}$ implies that $\chi^{p,q}_k = \chi^{p-k,q-k} - \chi^{p-k-1,q-k-1}$ and thus we can assume $k=0.$ Hence it suffices to show that
\begin{align*}
\chi^{p,q} - \chi^{p-1,q-1} = \frac{1}{\Delta} \cdot |\varepsilon^{n}, \varepsilon^{n-1}, \ldots, \widehat{\varepsilon^{n-p}}, \ldots, \widehat{\varepsilon^{q-1}}, \ldots,  1, \varepsilon^{-1}|.
\end{align*}

To this end, let $\sigma_k = \sum_{I_k} \varepsilon_{I_k}$ 
denote the $k$-th elementary symmetric polynomial in $\varepsilon_1, \ldots, \varepsilon_n$, with the convention that $\sigma_k=0$ if $k<0$, and set 
\begin{align*}
\tau_{a,b}= \sigma_{n-a+1} \sigma_{n-b} - \sigma_{n-a} \sigma_{n-b+1}.
\end{align*}

Notice that $\tau_{a,b}=-\tau_{b,a}.$ Computing the Vandermonde determinant 
\begin{align*}
P(s,t, \varepsilon_1, \ldots, \varepsilon_n) = \begin{vmatrix}
s^{n+1} & s^n & \ldots & 1 \\
t^{n+1} & t^n & \ldots & 1 \\
\varepsilon_1^{n+1} & \varepsilon_1^n  & \ldots & 1 \\
\vdots & \vdots & & \vdots \\
\varepsilon_n^{n+1} & \varepsilon_n^n  & \ldots & 1 
\end{vmatrix}
\end{align*}
as a difference product we obtain
\begin{align*}
P(s,t, \varepsilon_1, \ldots, \varepsilon_n)  = & \ \Delta \cdot (s-t) \cdot \prod_{i=1}^n (s-\varepsilon_i) \cdot \prod_{i=1}^n (t-\varepsilon_i) \\
= & \ \Delta \cdot (s-t) \cdot 
\left( \sum_{k=0}^n (-1)^{k} \sigma_k s^{n-k} \right) \cdot \left( \sum_{k=0}^n (-1)^{k} \sigma_k t^{n-k} \right) \\
= & \ \Delta \cdot (s-t) \cdot \sum_{a,b=0}^n (-1)^{a+b} s^a t^b \sigma_{n-a} \sigma_{n-b} \\
= & \ \Delta \cdot \sum_{a=1}^{n+1} \sum_{b=0}^n (-1)^{a+b+1} s^a t^b \sigma_{n-a+1} \sigma_{n-b} \\
& \ - \Delta \cdot \sum_{a=0}^{n} \sum_{b=1}^{n+1} (-1)^{a+b+1} s^a t^b \sigma_{n-a} \sigma_{n-b+1} \\
= & \ \Delta \cdot \sum_{a,b=0}^{n+1} (-1)^{a+b+1} s^a t^b \tau_{a,b} = \Delta \cdot \sum_{a<b} (-1)^{a+b+1} (s^b t^a - s^a t^b) \tau_{b,a}.
\end{align*}

On the other hand, Laplace expansion along the first two rows, cf.\cite[Theorem 93]{MuirTreatiseOnDeterminants}, yields
\begin{align*}
P(s,t, \varepsilon_1, \ldots, \varepsilon_n)  = \sum_{a<b} (-1)^{a+b+1} \begin{vmatrix}
s^b & s^a \\
t^b & t^a 
\end{vmatrix}
| \varepsilon^{n+1}, \ldots, \widehat{\varepsilon^b}, \ldots, \widehat{\varepsilon^a}, \ldots, \varepsilon, 1 |.
\end{align*}

It follows that
\begin{align*}
\tau_{b,a} = \frac{1}{\Delta} | \varepsilon^{n+1}, \ldots, \widehat{\varepsilon^b}, \ldots, \widehat{\varepsilon^a}, \ldots, \varepsilon, 1 |
= \frac{\varepsilon_{I_n}}{\Delta} \cdot | \varepsilon^{n}, \ldots, \widehat{\varepsilon^{b-1}}, \ldots, \widehat{\varepsilon^{a-1}}, \ldots, 1, \varepsilon^{-1} |.
\end{align*}
The claim now follows from the computation
\begin{align*}
\chi^{p,q} - \chi^{p-1,q-1} = & \ \sum_{I_p, J_q} \varepsilon_{I_p} \bar{\varepsilon}_{J_q} - \sum_{\stackrel{I_{p-1},}{J_{q-1}}} \varepsilon_{I_{p-1}} \bar{\varepsilon}_{J_{q-1}} =  \sum_{I_p, J_{n-q}} \varepsilon_{I_p} \frac{\varepsilon_{J_{n-q}}}{\varepsilon_{I_n}} - \sum_{\stackrel{I_{p-1},}{J_{n-q+1}}} \varepsilon_{I_{p-1}} \frac{\varepsilon_{J_{n-q+1}}}{\varepsilon_{I_n}}  \\
= & \ \frac{1}{\varepsilon_{I_n}} \left( \sigma_p \sigma_{n-q} - \sigma_{p-1} \sigma_{n-q+1} \right) = \frac{\tau_{n-p+1,q}}{\varepsilon_{I_n}} \\
= & \  \frac{1}{\Delta} | \varepsilon^{n}, \ldots, \widehat{\varepsilon^{n-p}}, \ldots, \widehat{\varepsilon^{q-1}}, \ldots, 1, \varepsilon^{-1} |.
\end{align*}
\end{proof}

\section{Estimating the curvature term of the Lichnerowicz Laplacian}
\label{SectionEstimationCurvatureTerm}

We continue to study $(p,q)$-forms on a Euclidean vector space $(V,g)$ with a compatible almost complex structure. Let $n=\dim_{\C} V.$

Based on lemma \ref{WeightedGeneralBochnerLemma}, we can control the curvature term of the Lichnerowicz Laplacian by estimating $| L \varphi|^2$ for $L \in \mathfrak{u}(V)$ and by calculating $|\varphi{^{\mathfrak{u}}}|^2$. This relies on the $U(n)$-irreducible decomposition of $\Ext^{p,q}V^{*}.$

\begin{definition}
\label{DefinitionKaehlerFormReducedForms}
For $\varphi \in \Ext^{p,q} V^{*}$ set
\begin{align*}
\mathring{\varphi}= 
\begin{cases} 
\varphi - \frac{g (\varphi,\omega^{p})}{| \omega^{p }|} \omega^{p} & \text{ if } p=q, \\
\hspace{10mm} \varphi & \text{ if } p \neq q.
\end{cases}
\end{align*}
\end{definition}

Notice that $L \varphi = L \mathring{\varphi}$ for all $L \in \mathfrak{u}(V).$

\begin{proposition}
\label{NormHatpqForm}
Let $k \leq \min \lbrace p, q\rbrace$ and $\varphi \in \Ext^{p,q}_k V^{*}.$ It follows that
\begin{align*}
| \varphi{^{\mathfrak{u}}} |^2 = \left( 2 (p-k)(q-k) + (p+q-2k)(n+1-(p+q-2k)) \right) | \mathring{\varphi} |^2.
\end{align*}
\end{proposition}
\begin{proof}
For notational simplicity replace $(p,q)$ by $(p+k,q+k)$. Recall from section \ref{SubsectionOnTensorsAndHat} that
\begin{align*}
| \varphi{^{\mathfrak{u}}}|^2 = \sum_{\Xi_{\alpha} \in \mathfrak{u}(V)} | \Xi_{\alpha} \varphi |^2.
\end{align*}
For the computation, we will use the explicit orthonormal basis $\lbrace R_{ij}, I_{ij}, I_{ii} \rbrace$ for $\mathfrak{u}(V)$ given in section \ref{SubsectionLieAlgebraUn}. 
Notice that due to Schur's lemma, it suffices to consider 
\begin{align*}
\varphi = dz^{1} \wedge \ldots \wedge dz^p \wedge \omega^k \wedge d \bar{z}^{p+1} \wedge  \ldots \wedge d \bar{z}^{p+q}  \in \Ext^{p+k,q+k}_k V^{*}.
\end{align*}

In case $p=q=0$, we have $\varphi = \omega^k$ and thus $\varphi{^\mathfrak{u}} = 0$ due to remark \ref{KaehlerFormInKernel}. Therefore, we may assume $p>0$ or $q>0$. It follows that $\mathring{\varphi} = \varphi.$

It is immediate from proposition \ref{ActionOfLieAlgebraUnOnPQForms} and remark \ref{KaehlerFormInKernel} that
\begin{align*}
| I_{ii} \varphi |^2 = \begin{cases}
| \varphi |^2 & \text{ for } 1 \leq i \leq p+q, \\
0 & \text{otherwise}.
\end{cases}
\end{align*}
Similarly, 
\begin{align*}
| R^{ij} \varphi |^2 = | I^{ij} \varphi |^2 
= \begin{cases}
| \varphi |^2 & \text{ for } i \in \lbrace 1, \ldots, p \rbrace,
\ j \in \lbrace p+1, \ldots, p+q \rbrace, \\
\frac{| \varphi |^2}{2} & \text{ for } i \in \lbrace 1, \ldots, p \rbrace, 
\ j \in \lbrace p+q+1, \ldots, n \rbrace, \\
\frac{| \varphi |^2}{2} & \text{ for } i \in \lbrace p+1, \ldots,  p+q \rbrace, 
\ j \in \lbrace p+q+1, \ldots, n \rbrace, \\
0 & \text{otherwise}
\end{cases}
\end{align*}
is straightforward unless $i \in \lbrace 1, \ldots, p \rbrace$ and $j \in \lbrace p+1, \ldots, p+q \rbrace.$ To check this remaining case, it suffices to compute $|R_{p,p+1} \varphi|^2.$ 

Observe that the K\"ahler form $\omega$ satisfies
\begin{align*}
\omega^k = \left( \sqrt{-1} \right)^k \frac{k!}{2^k}  \sum_{1 \leq i_1 < \ldots < i_k \leq n} dz^{i_1} \wedge d \bar{z}^{i_1} \wedge \ldots \wedge dz^{i_k} \wedge d \bar{z}^{i_k}
\end{align*}
and thus 
\begin{align*}
\varphi = \left( \sqrt{-1} \right)^k \frac{k!}{2^k} \cdot  dz^{1} \wedge & \ldots \wedge dz^p \wedge d \bar{z}^{p+1} \wedge \ldots \wedge d \bar{z}^{p+q} \wedge \\
& \wedge  \sum_{p+q+1 \leq i_1 < \ldots < i_k \leq n} dz^{i_1} \wedge d \bar{z}^{i_1} \wedge \ldots \wedge dz^{i_k} \wedge d \bar{z}^{i_k}.
\end{align*}
Using proposition \ref{ActionOfLieAlgebraUnOnPQForms} and remark \ref{KaehlerFormInKernel} again we find that 
\begin{align*}
\left(- \sqrt{-1} \right)^k \frac{2^k}{k!} \cdot & \sqrt{2} R_{p, p+1} \varphi  = \\
= & \ - dz^{1} \wedge \ldots \wedge dz^{p-1} \wedge d z^{p+1} \wedge d \bar{z}^{p+1} \wedge \ldots \wedge d \bar{z}^{p+q} \wedge  \left(- \sqrt{-1} \right)^k \frac{2^k}{k!} \cdot \omega^k \\
& \ +  dz^{1} \wedge \ldots \wedge dz^{p}  \wedge d \bar{z}^{p} \wedge d \bar{z}^{p+2} \wedge \ldots \wedge d \bar{z}^{p+q} \wedge  \left(- \sqrt{-1} \right)^k \frac{2^k}{k!} \cdot \omega^k \\
= & \ - dz^{1} \wedge \ldots \wedge dz^{p-1} \wedge d z^{p+1} \wedge  d \bar{z}^{p+1} \wedge \ldots \wedge d \bar{z}^{p+q} \wedge  \\
& \hspace{10mm} \wedge dz^{p} \wedge d \bar{z}^{p} \wedge  \sum_{p+q+1 \leq i_2 < \ldots < i_k \leq n} dz^{i_2} \wedge d \bar{z}^{i_2} \wedge \ldots \wedge dz^{i_k} \wedge d \bar{z}^{i_k} \\
& +  \ dz^{1} \wedge \ldots \wedge dz^{p} \wedge d \bar{z}^{p} \wedge d \bar{z}^{p+2} \wedge \ldots \wedge d \bar{z}^{p+q} \\ 
& \hspace{10mm} \wedge dz^{p+1} \wedge d \bar{z}^{p+1} \wedge \sum_{p+q+1 \leq i_2 < \ldots < i_k \leq n} dz^{i_2} \wedge d \bar{z}^{i_2} \wedge \ldots \wedge dz^{i_k} \wedge d \bar{z}^{i_k}  \\
& \ + dz^{1} \wedge \ldots \wedge dz^{p-1} \wedge \left( dz^{p} \wedge d \bar{z}^{p} - dz^{p+1} \wedge d \bar{z}^{p+1}  \right) \wedge d \bar{z}^{p+2} \wedge \ldots \wedge d \bar{z}^{p+q} \wedge \\& \hspace{10mm} \wedge \sum_{p+q+1 \leq i_1 < \ldots < i_k \leq n} dz^{i_1} \wedge d \bar{z}^{i_1} \wedge \ldots \wedge dz^{i_k} \wedge d \bar{z}^{i_k}  \\
= & \ dz^{1} \wedge \ldots \wedge dz^{p-1} \wedge \left( dz^{p} \wedge d \bar{z}^{p} - dz^{p+1} \wedge d \bar{z}^{p+1}  \right) \wedge d \bar{z}^{p+2} \wedge \ldots \wedge d \bar{z}^{p+q} \wedge \\& \hspace{10mm} \wedge \sum_{p+q+1 \leq i_1 < \ldots < i_k \leq n} dz^{i_1} \wedge d \bar{z}^{i_1} \wedge \ldots \wedge dz^{i_k} \wedge d \bar{z}^{i_k}.
\end{align*}
Thus $\sqrt{2} R_{p,p+1} \varphi$ is the difference of two orthogonal forms, both of which have the same norm as $\varphi.$ Hence, $|R_{p,p+1} \varphi|^2 = |\varphi|^2$ as claimed. 

Overall we obtain
\begin{align*}
| \varphi{^{\mathfrak{u}}} |^2 & = \left( (p+q) + 2 pq + p(n-(p+q)) +  q(n-(p+q)) \right) | \varphi |^2 \\
& = \left(  2pq + (p+q)(n+1-(p+q)) \right)  | \mathring{\varphi} |^2.
\end{align*}
\end{proof}

\begin{corollary}
A $(p,q)$-form $\varphi$ satisfies $|\varphi{^{\mathfrak{u}}}|^2 = 0$ if and only if $\mathring{\varphi}=0,$ i.e. $\varphi=0$ or $p=q$ and $\varphi$ is a multiple of $\omega^p.$ 
\label{HatZeroOnlyOnMultiplesOfKaehlerForm}
\end{corollary}
\begin{proof}
This is immediate from the orthogonal decomposition of $\Ext^{p,q} V^{*}$ into $U(n)$-irreducible components in theorem \ref{IrreducibleUnDecomposition} and the characterization of $|\varphi{^{\mathfrak{u}}}|^2$ in proposition \ref{NormHatpqForm}.
\end{proof}

\begin{proposition}
Suppose that $\varphi \in V^{p,q}_k.$ It follows that
\begin{align*}
| L \varphi |^2 \leq (p+q-2k) |L|^2 |\mathring{\varphi}|^2
\end{align*}
for all $L \in \mathfrak{u}(V).$
\label{EstimateLieAlgebraActionpqForms}
\end{proposition}
\begin{proof}
For a given $L \in \mathfrak{u}(V)$ there is an orthonormal basis for $V$ that puts $L$ in its normal form. In particular, there are $\mu_1, \ldots, \mu_n \in \R$ such that
\begin{align*}
L= \sum_{i=1}^n \mu_i I_{ii}.
\end{align*}

Consider 
\begin{align*}
\Phi^{I_{p-k},J_{q-k}} = dz^{I_{p-k}} \wedge \omega^k \wedge d \bar{z}^{J_{q-k}} \in V^{p,q}_k.
\end{align*}
According to proposition \ref{ActionOfLieAlgebraUnOnPQForms} and remark \ref{KaehlerFormInKernel} we have 
\begin{align*}
(L)  \Phi^{I_{p-k},J_{q-k}} & =  \sqrt{-1} \ \left( \ \sum_{i \in I_{p-k}} \mu_i - \sum_{j \in J_{q-k}} \mu_j \ \right) \ \Phi^{I_{p-k},J_{q-k}} \\
& = \sqrt{-1} \ \left( \ \sum_{i \in I_{p-k} \setminus J_{q-k}} \mu_i - \sum_{j \in J_{q-k} \setminus I_{p-k}} \mu_j \ \right) \ \Phi^{I_{p-k},J_{q-k}}.
\end{align*}
This directly implies
\begin{align*}
| (L)  \Phi^{I_{p-k},J_{q-k}} |^2 \leq (p+q-2k) |L|^2 | \mathring{\Phi}^{I_{p-k},J_{q-k}}|^2.
\end{align*}

For an arbitrary $\varphi \in V^{p,q}_k$, note that there are $\lambda_{I_{p-k}, J_{q-k}} \in \C$ such that
\begin{align*}
\varphi = \sum_{I_{p-k}, J_{q-k}} \lambda_{I_{p-k}, J_{q-k}} \Phi^{I_{p-k},J_{q-k}}.
\end{align*}

The claim now follows from the above computation and the observation that $\Phi^{I_{p-k},J_{q-k}}$ and $\Phi^{\tilde{I}_{p-k}, \tilde{J}_{q-k}}$ are orthogonal unless $I_{p-k}= \tilde{I}_{p-k}$ and $J_{q-k}=\tilde{J}_{q-k}$. 
\end{proof}

For $k \leq \min \lbrace p, q \rbrace$ notice that $p+q-2k = 0$  if and only if $p=q=k.$ In case $p+q-2k \neq 0$ set 
\begin{align*}
C^{p,q}_k & = \frac{2 (p-k)(q-k) + (p+q-2k)(n+1-(p+q-2k)) }{(p+q-2k)} \\
& = n+1-(p+q)+2 \frac{pq-k^2}{p+q-2k}.
\end{align*}
Note that $C^{p,p}_k = n+1-p+k$.

We can now estimate the curvature term of the Lichnerowicz Laplacian on $\Ext^{p,q}_k V^{*}.$

\begin{remark} \normalfont
\label{LichnerowiczLaplacianVanishesOnKaehlerForm} On $\Ext^{p,p}_p V^{*} = \operatorname{span}_{\C} \lbrace \omega^p \rbrace$ we have $g(\mathfrak{R}(\varphi{^{\mathfrak{u}}}), \overline{\varphi}{^{\mathfrak{u}}}) = 0$ due to remark \ref{KaehlerFormInKernel}.
\end{remark}

\begin{proposition}
Let $k \leq \min \lbrace p, q \rbrace$ with $p+q-2k > 0$. Let $\kappa \leq 0$ and let $\mathfrak{R} \colon \mathfrak{u}(V) \to \mathfrak{u}(V)$ be a K\"ahler curvature operator with eigenvalues $\lambda_1 \leq \ldots \leq \lambda_{n^2}$. Let $\varphi \in \Ext^{p,q}_k V^{*}$.

If
\begin{align*}
\lambda_1 + \ldots + \lambda_{\floor{C^{p,q}_k}} + \left( C^{p,q}_k - \floor{C^{p,q}_k} \right) \cdot  \lambda_{\floor{C^{p,q}_k}+1} \geq \kappa \left( \floor{C^{p,q}_k}+1 \right), 
\end{align*}
then 
\begin{align*}
g(\mathfrak{R}(\varphi{^{\mathfrak{u}}}), \overline{\varphi}{^{\mathfrak{u}}}) \geq \kappa \left( \floor{C^{p,q}_k}+1 \right) (p+q-2k) | \mathring{\varphi}|^2.
\end{align*}

If
\begin{align*}
\lambda_1 + \ldots + \lambda_{\floor{C^{p,q}_k}} + \left( C^{p,q}_k - \floor{C^{p,q}_k} \right) \cdot  \lambda_{\floor{C^{p,q}_k}+1} > 0,
\end{align*}
then $g(\mathfrak{R}(\varphi{^{\mathfrak{u}}}), \overline{\varphi}{^{\mathfrak{u}}}) >0$ unless $\varphi = 0$.
\label{LichnerowiczEstimatePQForms}
\end{proposition}
\begin{proof}
Propositions \ref{NormHatpqForm} and \ref{EstimateLieAlgebraActionpqForms} imply that
\begin{align*}
|L \varphi|^2 \leq (p+q-2k) |L|^2 | \mathring{\varphi} |^2 = \frac{1}{C^{p,q}_k} |L|^2 | \varphi {^{\mathfrak{u}}} |^2.
\end{align*}

Lemma \ref{WeightedGeneralBochnerLemma} yields
\begin{align*}
g(\mathfrak{R}(\varphi{^{\mathfrak{u}}}), \overline{\varphi}{^{\mathfrak{u}}}) \geq \frac{ \kappa \left( \floor{C^{p,q}_k}+1 \right)}{C^{p,q}_k} | \varphi{^{\mathfrak{u}}}|^2
\end{align*}
and proposition \ref{NormHatpqForm} shows that $g(\mathfrak{R}(\varphi{^{\mathfrak{u}}}), \overline{\varphi}{^{\mathfrak{u}}}) \geq \kappa \left( \floor{C^{p,q}_k}+1 \right) (p+q-2k) | \mathring{\varphi}|^2.$ 

Notice that in fact $\mathring{\varphi} = \varphi$ since $p+q-2k >0$. In particular, $\varphi$ cannot be a non-zero multiple of the K\"ahler form. Hence last claim follows from lemma \ref{WeightedGeneralBochnerLemma} and corollary \ref{HatZeroOnlyOnMultiplesOfKaehlerForm}.
\end{proof}

By imposing the strongest curvature assumption in proposition \ref{LichnerowiczEstimatePQForms}, we obtain a uniform estimate for all $\varphi \in \Ext^{p,q}_k V^{*}$ by estimating $ g(\mathfrak{R}(\varphi{^{\mathfrak{u}}}),\overline{\varphi}{^{\mathfrak{u}}}) \geq \kappa C(n,p,q) | \mathring{\varphi}|^2$ with a constant $C(n,p,q)$ independent of $k.$ More precisely, with the constants $C^{p,q}=C^{p,q}_0$ defined in the introduction, we have

\begin{corollary}
\label{StrongestCurvatureAssumption}
Let $\kappa \leq 0$ and $\varphi \in \Ext^{p,q}_k V^{*}$. If 
\begin{align*}
\lambda_1 + \ldots + \lambda_{\floor{C^{p,q}}} + \left( C^{p,q} - \floor{C^{p,q}} \right) \cdot \lambda_{\floor{C^{p,q}}+1} \geq \kappa (\floor{C^{p,q}}+1),
\end{align*}
then 
\begin{align*}
g(\mathfrak{R}(\varphi{^{\mathfrak{u}}}), \overline{\varphi}{^{\mathfrak{u}}}) \geq \kappa \left( n+2-|p-q| \right) (p+q) | \mathring{ \varphi} |^2.
\end{align*}
\end{corollary}
\begin{proof}
Due to remark \ref{LichnerowiczLaplacianVanishesOnKaehlerForm} the estimate is clearly valid if $p=q=k.$ For $0 \leq k \leq \min \lbrace p, q \rbrace$ with $p+q-2k > 0$ the function
\begin{align*}
\frac{pq-k^2}{p+q-2k}
\end{align*}
takes values 
\begin{align*}
\frac{pq}{p+q} \leq \ldots \leq \min \lbrace p ,q \rbrace
\end{align*}
and thus 
\begin{align*}
C^{p,q}_0
 \leq \ldots \leq C^{p,q}_k \leq \ldots \leq 
C^{p,q}_{\min \lbrace p, q \rbrace}.
\end{align*}
Notice that $C^{p,q}_0 = n+1-\frac{p^2+q^2}{p+q}=C^{p,q}$ and $C^{p,q}_{\min \lbrace p, q \rbrace}= n+1 - | p-q|.$

Therefore, by assumption, the curvature conditions in proposition \ref{LichnerowiczEstimatePQForms} are satisfied for each module $\Ext^{p,q}_k V^{*}$ individually. Thus for every $\varphi \in \Ext^{p,q}_k V^{*}$ we have
\begin{align*}
g(\mathfrak{R}(\varphi{^{\mathfrak{u}}}), \overline{\varphi}{^{\mathfrak{u}}}) \geq \kappa \left( \floor{C^{p,q}_k}+1 \right) (p+q-2k) | \mathring{\varphi}|^2 \geq \kappa \left( n+2-|p-q| \right) (p+q) | \mathring{ \varphi} |^2.
\end{align*}
\end{proof}

\begin{remark} \normalfont
(a) Note that $C^{p,q}$ is minimal if $(p,q)=(n,0)$ or $(0,n)$ and $C^{n,0}=1.$ Furthermore, $C^{n-1,0}=2$ and $C^{n-1,1}_0 = 3 - \frac{2}{n}$ but $C^{n-1,1}_1=3.$ Thus, for a $3$-nonnegative K\"ahler curvature operator, corollary \ref{StrongestCurvatureAssumption} establishes nonnegativity of the curvature term on $\Ext^{p,q}_k V^{*}$ unless $k=0$ and $(p,q)=(n,0),$ $(n-1,0)$ or $(n-1,1).$

(b) If $C^{p,q}$ is an integer, e.g. $C^{p,p}=n+1-p,$ it is more natural to assume that
\begin{align*}
\lambda_1 + \ldots + \lambda_{C^{p,q}} \geq \kappa C^{p,q}.
\end{align*}

As in corollary \ref{StrongestCurvatureAssumption} it follows that 
\begin{align*}
g(\mathfrak{R}(\varphi{^{\mathfrak{u}}}), \overline{\varphi}{^{\mathfrak{u}}}) \geq \kappa  C^{p,q}_k (p+q-2k) | \mathring{ \varphi} |^2 \geq  \kappa  \left( 2pq + (n+1-(p+q))(p+q) \right) | \mathring{ \varphi} |^2
\end{align*}
for every $\varphi \in \Ext^{p,q}_k V^{*}.$
\end{remark}

\section{The Lichnerowicz Laplacian on $(p,q)$-forms}
\label{SectionProofOfMainTheorems}

In this section we prove Theorems \ref{MainVanishingTheorem} - \ref{MainEstimationTheorem}. Theorem \ref{MainVanishingTheorem} is a direct consequence of Theorem \ref{VanishingHodgeNumbers} and Bochner's result \cite{BochnerVectorFieldsAndRic} that every K\"ahler manifold with positive Ricci curvature satisfies $h^{n-1,0}=h^{n,0}=0.$ \vspace{2mm} \\
\textit{Proof} of Theorems \ref{VanishingHodgeNumbers} - \ref{MainEstimationTheorem}. Due to the K\"ahler identities and Hodge's theorem, we may study the space of harmonic $(p,q)$-forms with respect to the Hodge Laplacian. We consider the Hodge Laplacian as a Lichnerowicz Laplacian as in example \ref{HodgeLaplacianAsLichnerowicz}.

Recall that every harmonic $(p,q)$-form $\varphi$ satisfies
\begin{align*}
\Delta \frac{1}{2} |\varphi|^2 = | \nabla \varphi |^2 + g( \Ric(\varphi), \overline{\varphi}).
\end{align*}

According to theorem \ref{IrreducibleUnDecomposition}, the decomposition of the space of $(p,q)$-forms into orthogonal, $U(n)$-irreducible modules is given by 
\begin{align*}
\Ext^{p,q} T^{*}M = \bigoplus_{k=0}^{\min \lbrace p,q \rbrace} \Ext^{p,q}_k T^{*}M.
\end{align*}

Recall from section \ref{SectionHolonomy} that the curvature term of the Lichnerowicz Laplacian preserves the irreducible decomposition, 
\begin{align*}
\Ric_{ |\Ext^{p,q}_k T^{*}M} \ \colon \Ext^{p,q}_k T^{*}M \to \Ext^{p,q}_k T^{*}M.
\end{align*}
For $\varphi \in \Ext^{p,q} T^{*}M$ there are $\varphi_k \in \Ext^{p,q}_k TM$ such that
\begin{align*}
\varphi = \varphi_0 + \ldots + \varphi_{\min \lbrace p,q \rbrace}.
\end{align*}

The above discussion and proposition \ref{CurvatureTermLichnerowiczLaplacian} imply that
\begin{align*}
g( \Ric( \varphi), \overline{\varphi}) & = \sum_{k=0}^{\min \lbrace p,q \rbrace} g( \Ric( \varphi_k), \overline{\varphi}_k) = \sum_{k=0}^{\min \lbrace p,q \rbrace} g(\mathfrak{R}((\varphi_k){^{\mathfrak{u}}}), (\overline{\varphi}_k){^{\mathfrak{u}}}).
\end{align*}
Let $\kappa \leq 0.$ Corollary \ref{StrongestCurvatureAssumption} shows that if
\begin{align*}
\lambda_1 + \ldots + \lambda_{\floor{C^{p,q}}} + \left( C^{p,q} - \floor{C^{p,q}} \right) \cdot \lambda_{\floor{C^{p,q}}+1} \geq \kappa (\floor{C^{p,q}}+1),
\end{align*}
then
\begin{align*}
g( \Ric( \varphi), \overline{\varphi}) & \geq  \kappa \left( n+2-|p-q| \right) (p+q) \sum_{k=0}^{\min \lbrace p,q \rbrace} | \mathring{\varphi_k} |^2  = \kappa \left( n+2-|p-q| \right) (p+q)  | \mathring{\varphi} |^2.
\end{align*}

Theorem \ref{MainEstimationTheorem} now follows directly from the Bochner technique as developed by  P. Li \cite{LiSobolevConstant} and Gallot \cite{GallotSobolevEstimates}, cf. \cite[Theorem 1.9]{PetersenWinkNewCurvatureConditionsBochner}.

If $\kappa = 0,$ then $g( \Ric( \varphi), \overline{\varphi}) \geq 0$ together with the maximum principle immediately imply Theorem \ref{MainRigidityTheorem}. 

Finally, for Theorem \ref{VanishingHodgeNumbers}, suppose that 
\begin{align*}
\lambda_1 + \ldots + \lambda_{\floor{C^{p,q}}} + \left( C^{p,q} - \floor{C^{p,q}} \right) \cdot \lambda_{\floor{C^{p,q}}+1} > 0.
\end{align*}

By Theorem \ref{MainRigidityTheorem}, every harmonic $(p,q)$-form $\varphi$ is parallel. Moreover, remark \ref{LichnerowiczLaplacianVanishesOnKaehlerForm} and proposition \ref{LichnerowiczEstimatePQForms} show that 
\begin{align*}
g( \Ric( \varphi), \overline{\varphi}) > 0
\end{align*}
unless $\varphi = 0$ or $\varphi$ is a multiple of a power of the K\"ahler form.  \hfill $\Box$

\vspace{2mm}

The following example shows that the curvature assumptions in Theorem \ref{MainVanishingTheorem} are different from positive orthogonal bisectional curvature.

\begin{example} \normalfont
Consider the basis $\Xi_{1,\pm} = \frac{1}{\sqrt{2}} \left( e_1 \wedge e_2 \pm e_3 \wedge e_4 \right),$ $\Xi_{2,\pm}  = \frac{1}{\sqrt{2}} \left( e_1 \wedge e_3 \pm e_4 \wedge e_2 \right),$ $\Xi_{3, \pm}  = \frac{1}{\sqrt{2}} \left(  e_1 \wedge e_4 \pm e_2 \wedge e_3 \right)$ for $\Ext^2\R^4$. Note that $\lbrace \Xi_{1,+}, \Xi_{1,-}, \Xi_{2,-}, \Xi_{3,-} \rbrace$ is a basis for $\mathfrak{u}(2) \subset \mathfrak{so}(4).$

Let $\varepsilon>0$ and set $\mu_{1,+}=6,$ $\mu_{2,+}=\mu_{3,+}=0$ and $\mu_{1,-}=6+ 2 \varepsilon,$ $\mu_{2,-}=\mu_{3,-}=-\varepsilon.$ It follows that the operator $\mathfrak{R} \colon \Ext^2 \R^4 \to \Ext^2 \R^4$ defined by $\mathfrak{R} ( \Xi_{i, \pm} ) = \mu_{i, \pm}  \Xi_{i, \pm}$ is a K\"ahler-Einstein algebraic curvature operator, cf. \cite[Example 4.3]{PetersenWinkNewCurvatureConditionsBochner}.

Note that $\lambda_1 = \mu_{2,-}$, $\lambda_2 = \mu_{3,-}$, $\lambda_3=\mu_{1,+}$ and $\lambda_4 = \mu_{1,-}$ are the eigenvalues of the associated K\"ahler curvature operator. In particular, for every $\alpha>0$ there is $\varepsilon>0$ such that $\lambda_1 + \lambda_2 + \alpha \lambda_3 > 0$ while $\lambda_1+\lambda_2<0.$ 

Wilking \cite{WilkingALieAlgebriaicApproach} observed that a K\"ahler curvature operator $\mathfrak{R} \colon \mathfrak{u}(2) \to \mathfrak{u}(2)$ has nonnegative orthogonal bisectional curvature if and only if it has nonnegative isotropic curvature. In the above example we have $R_{1313}=R_{1414}=R_{2323}=R_{2424}=- \frac{\varepsilon}{2}$ and $R_{1234}= -\varepsilon.$ In particular, $\mathfrak{R}$ has negative isotropic curvatures.
\label{EigenvalueConditionIndependentForOrthBisectCurv}
\end{example}

\section{A Tachibana Theorem for K\"ahler manifolds}
\label{SectionKaehlerTachibana}

\begin{proposition}
The curvature tensor $R \in \operatorname{Sym}_B^2( \mathfrak{u}(n))$ of 
\begin{align*}
\mathbb{CP}^k \times \C^{n-k}
\end{align*}
satisfies $|R{^{\mathfrak{u}}}|^2 = 32 k (k+1) (n-k).$

In particular, the curvature tensor of $\mathbb{CP}^n$ satisfies $| \left( R_{\mathbb{CP}^n} \right){^{\mathfrak{u}}}|^2=0.$
\label{CPkxCnkExample}
\end{proposition}
\begin{proof}
We may assume $k>0.$ We will pick an orthonormal eigenbasis $\lbrace \Xi_{\alpha} \rbrace$ for the K\"ahler curvature operator so that the eigenvectors $\Xi_{\alpha}$ correspond to the $\mathbb{CP}^k$-factor for $\alpha=1, \ldots, k^2$ and to the $\C^{n-k}$-factor for $\alpha=(n-k)^2 + 1, \ldots, n^2.$ Specifically we consider 
\begin{align*}
R_{ij} & = \frac{1}{\sqrt{2}} \left( e_i \wedge e_j + f_i \wedge f_j \right) \ \text{ for } 1 \leq i < j \leq n, \\
I_{ij} & = \frac{1}{\sqrt{2}} \left( e_i \wedge f_j + e_j \wedge f_i \right) \ \text{ for } 1 \leq i < j \leq n, \\
S_i & = \frac{1}{\sqrt{i+i^2}} \left(  - i e_{i+1} \wedge f_{i+1} + \sum_{j=1}^i e_j \wedge f_j \right) \ \text{ for } 1 \leq i \leq k-1, \\
\Xi_{k^2} & = \frac{1}{\sqrt{k}} \sum_{i=1}^k e_i \wedge f_i, \\
I_{ii} & = e_i \wedge f_i  \ \text{ for } k+1 \leq i \leq n.
\end{align*}

In particular, $\lbrace \Xi_1, \ldots, \Xi_{k^2-1} \rbrace = \lbrace R_{ij}, I_{ij} \ \vert \ 1 \leq i < j \leq k \rbrace \cup \lbrace S_i  \ \vert \ 1 \leq i \leq k-1 \rbrace $ is an orthonormal basis for the eigenspace corresponding to the eigenvalue $\lambda_{\alpha}=2,$ the normalized K\"ahler form $\Xi_{k^2}$ of the $\mathbb{CP}^k$-factor spans the eigenspace corresponding to the eigenvalue $\lambda_{k^2}=2(k+1)$ and all other eigenvectors lie in the kernel.

Recall from example \ref{HatsOfCurvatureTensors} that
\begin{align*}
| R{^{\mathfrak{u}}} |^2 = 2  \sum_{ \alpha < \beta} \sum_{\gamma} \left( \lambda_{\alpha} - \lambda_{\beta} \right)^2 g( ( \Xi_{\gamma} ) \Xi_{\alpha}, \Xi_{\beta} )^2
\end{align*}
and that $g( ( \Xi_{\gamma} ) \Xi_{\alpha}, \Xi_{\beta} )^2$ is fully symmetric in $\Xi_{\alpha},$ $\Xi_{\beta},$ $\Xi_{\gamma}.$

It suffices to consider $\alpha \in \lbrace 1, \ldots, k^2 \rbrace.$ This follows from the fact that if $k^2 < \alpha < \beta$ then $\lambda_{\alpha}=\lambda_{\beta}=0$ and thus these terms do not contribute to $| R{^{\mathfrak{u}}}|^2.$

In addition, we can assume $\beta \in \lbrace k^2+1, \ldots, n^2 \rbrace$ since $\left( \lambda_{\alpha} - \lambda_{\beta} \right) g( ( \Xi_{\gamma} ) \Xi_{\alpha}, \Xi_{\beta} ) = 0$ whenever $\alpha, \beta \in \lbrace 1, \ldots, k^2 \rbrace$. Indeed, we can assume $\Xi_{\beta}=\Xi_{k^2}$ as otherwise $\lambda_{\alpha} = \lambda_{\beta}$. However, since $(\Xi_{\alpha}) \Xi_{k^2}=0$ due to remark \ref{KaehlerFormInKernel}, it follows that $g( ( \Xi_{\gamma} ) \Xi_{\alpha}, \Xi_{\beta} )^2 = g( ( \Xi_{\alpha} ) \Xi_{\beta}, \Xi_{\gamma} )^2 = 0.$ 

Similarly we can assume $\gamma \in \lbrace k^2+1, \ldots, n^2 \rbrace$. Otherwise $\Xi_{\alpha}, \Xi_{\gamma} \in \mathfrak{u}(k)$ and hence also $( \Xi_{\gamma} ) \Xi_{\alpha}  = [\Xi_{\gamma},\Xi_{\alpha}] \in \mathfrak{u}(k)$ while $\Xi_{\beta} \in \mathfrak{u}(k) ^{\perp} \subset \mathfrak{u}(n)$ for $\beta \in \lbrace k^2+1, \ldots, n^2 \rbrace$.

In fact, it suffices to consider $\beta, \gamma \in \lbrace k^2+1, \ldots, (n-k)^2 \rbrace$, i.e. that $\Xi_{\beta}, \Xi_{\gamma}$ correspond to mixed curvatures: by definition of the basis, $\Xi_{\alpha}$ and $\Xi_{\delta}$ do not have overlapping indices for $\alpha \in \lbrace 1, \ldots, k^2 \rbrace$ and $\delta \in \lbrace (n-k)^2+1, \ldots, n^2 \rbrace.$ This implies $( \Xi_{\alpha} ) \Xi_{\delta} = 0$.

Overall we conclude that
\begin{align*}
| R{^{\mathfrak{u}}} |^2 = 2 \sum_{\alpha=1}^{k^2} \sum_{\beta, \gamma=k^2+1}^{(n-k)^2} \lambda_{\alpha}^2 \ g( ( \Xi_{\beta} ) \Xi_{\gamma}, \Xi_{\alpha} )^2.
\end{align*}

Note that the projection of $( \Xi_{\beta} ) \Xi_{\gamma}$ onto $\mathfrak{u}(k) \subset \mathfrak{u}(n)$ can only be non-zero if $\Xi_{\beta}, \Xi_{\gamma}$ have a common index $a>k+1.$ All of these possibilities are given by
\begin{align*}
( R_{ia} ) R_{ja} & = \frac{1}{\sqrt{2}} R_{ij}, \hspace{8.3mm} ( I_{ia} ) I_{ja} = \frac{1}{\sqrt{2}} R_{ij},  \\
( R_{ia} ) I_{ja} & = - \frac{1}{\sqrt{2}} I_{ij}, \hspace{5mm} ( I_{ia} ) R_{ja} = \frac{1}{\sqrt{2}} I_{ij}, \\
( R_{ia} ) I_{ia} & = I_{aa} - I_{ii}, \hspace{5mm} ( I_{ia} ) R_{ia} = I_{ii} - I_{aa}
\end{align*}
where $1 \leq i, j \leq k$, $i \neq j$, and $k+1 \leq a \leq n.$ 

Notice the first four terms all give the same contribution to $| R{^{\mathfrak{u}}} |^2$ and $\lambda_{\alpha}=2$ in all cases. Since each term appears $k(k-1)(n-k)$-many times, these terms add up to $16 k (k-1)(n-k)$.

Furthermore, since $g(I_{aa} - I_{ii}, \Xi_{k^2})^2 = \frac{1}{k}$ and all other inner products with $\Xi_{k^2}$ vanish, the inner products of the last two terms with $\Xi_{\alpha}= \Xi_{k^2}$ contribute $16(k+1)^2(n-k)$ to $| R{^{\mathfrak{u}}}|^2$.

Finally, notice that 
\begin{align*}
g( I_{aa} - I_{ii}, S_j) = \begin{cases}
0 & i > j+1, \\
-\frac{j}{j+j^2} & i=j+1, \\
-\frac{1}{j+j^2} & i<j+1.
\end{cases}
\end{align*}

Since all $S_j$ are eigenvectors corresponding to the eigenvalue $\lambda_{\alpha} =2$ the contribution of the above terms amounts to 
\begin{align*}
16 (n-k) \left( \sum_{j=1}^{k-1} \sum_{i=1}^j \frac{1}{j+j^2} + \sum_{j=1}^{k-1} \frac{j^2}{j+j^2} \right) = 16(n-k)(k-1).
\end{align*}
Overall,
\begin{align*}
| R{^{\mathfrak{u}}} |^2 =16(n-k)(k(k-1) + (k+1)^2+k-1) = 32k(k+1)(n-k).
\end{align*}
\end{proof}

The computation of $| R {^{\mathfrak{u}}} |^2$ for a general K\"ahler curvature tensor $R \in \operatorname{Sym}_B^2( \mathfrak{u}(n))$ relies on the decomposition of $\operatorname{Sym}_B^2( \mathfrak{u}(n))$ into orthogonal, $U(n)$-irreducible components. Specifically, the space of K\"ahler curvature tensors decomposes into the orthogonal subspaces of K\"ahler curvature tensors with constant holomorphic sectional curvature, K\"ahler curvature tensors with trace-free Ricci curvature, and Bochner tensors. Due to a result of Alekseevski \cite{AlekseevskiiRiemannianSpacesExceptionalHol}, 
this decomposition is indeed $U(n)$-irreducible. Note that the Bochner tensor is the K\"ahler analogue of the Weyl tensor, cf. \cite{BochnerCurvatureAndBettiNumbersII}. 

In particular, every K\"ahler curvature tensor $R \in \operatorname{Sym}_B^2( \mathfrak{u}(n))$ decomposes as 
\begin{align*}
R = \frac{\scal}{4n(n+1)} R_{ \mathbb{CP}^n } + R_0 + B.
\end{align*}

As in proposition \ref{CPkxCnkExample}, we use the convention that the curvature tensor $R_{\mathbb{CP}^n}$ of the complex projective space with the Fubini Study metric satisfies $\scal ( R_{\mathbb{CP}^n} )= 4n(n+1).$ Furthermore, the trace-free Ricci part $R_0$ satisfies $| R_0 |^2 = \frac{2}{n+2} | \mathring{\Ric}|^2.$

For a K\"ahler curvature tensor $R \in \operatorname{Sym}_B^2( \mathfrak{u}(n))$ set
\begin{align*}
\mathring{R} = R - \frac{\scal}{4n(n+1)} R_{\mathbb{CP}^n}.
\end{align*}

Thus, a K\"ahler curvature tensor $R \in \operatorname{Sym}_B^2( \mathfrak{u}(n))$ has constant holomorphic sectional curvature if and only if $| \mathring{R}|^2 =0$. 

Furthermore, $| \left( R_{\mathbb{CP}^n} \right){^{\mathfrak{u}}}|^2=0$ implies that $L R_{\mathbb{CP}^n} = 0$ for all $L \in \mathfrak{u}(n)$ and thus we have $L R = L \mathring{R}$ for every $L \in \mathfrak{u}(n)$ and every $R \in \operatorname{Sym}_B^2( \mathfrak{u}(n)).$

\begin{lemma}
Every algebraic K\"ahler curvature tensor $R \in \operatorname{Sym}^2_B(\mathfrak{u}(n))$ satisfies
\begin{align*}
|R{^{\mathfrak{u}}}|^2 = 4(n+1) | \mathring{R}|^2 - 4 | \mathring{\Ric}|^2.
\end{align*}
In particular, $|R{^{\mathfrak{u}}}|^2 = 0$ if and only if $R$ has constant holomorphic sectional curvature.
\label{HatOnKaehlerCurvatureOperators}
\end{lemma}
\begin{proof}
Due to the $U(n)$-irreducibility of the decomposition of $\operatorname{Sym}^2_B( \mathfrak{u}(n))$, there are constants $a, b, c \in \R$ such that
\begin{align*}
| R{^{\mathfrak{u}}} |^2 = a \scal^2 + b | \Ric |^2 + c |R|^2
\end{align*}
for every algebraic K\"ahler curvature tensor $R \in \operatorname{Sym}^2_B( \mathfrak{u}(n)).$ 

Evaluation on the curvature tensors of $\mathbb{CP}^k \times \C^{n-k}$ yields $a=0,$ $b=-4$ and $c=4(n+1)$ due to proposition \ref{CPkxCnkExample}. 

It follows that 
\begin{align*}
|R{^{\mathfrak{u}}}|^2 = 4(n+1) |R|^2 - 4 |\Ric|^2 = 4(n+1) | \mathring{R}|^2 - 4 | \mathring{\Ric}|^2 = \frac{4n}{n+2} | \mathring{\Ric} |^2 + 4 (n+1) |B|^2.
\end{align*}

In particular, $R{^{\mathfrak{u}}}=0$ if and only if $\mathring{R}=0,$ which implies the claim. 
\end{proof}

\textit{Proof} of Theorem \ref{KaehlerTachibana}. The curvature tensor $R$ of an Einstein manifold is harmonic and thus satisfies the Bochner formula
\begin{align*}
\Delta \frac{1}{2} |R|^2 = | \nabla R|^2 + \frac{1}{2} \cdot g(  \Ric(R), \overline{R}).
\end{align*}

For algebraic K\"ahler curvature operators $R \in \operatorname{Sym}_B^2( \mathfrak{u}(n))$ it follows as in \cite[Lemma 2.2]{PetersenWinkNewCurvatureConditionsBochner} that
\begin{align*}
| L R |^2 = |L \mathring{R} |^2 \leq 8 |L|^2 | \mathring{R}|^2 
\end{align*}
for every $L \in \mathfrak{u}(n)$. In the K\"ahler-Einstein case, $\mathring{\Ric}=0$, lemma \ref{HatOnKaehlerCurvatureOperators} thus implies
\begin{align*}
| L R |^2  \leq \frac{2}{n+1} |L|^2 |R{^{\mathfrak{u}}}|^2
\end{align*}
for every $L \in \mathfrak{u}(n)$. Combined with proposition \ref{CurvatureTermLichnerowiczLaplacian} and lemma \ref{WeightedGeneralBochnerLemma}, the assumption
\begin{align*}
\lambda_{1} + \ldots + \lambda_{\floor{\frac{n+1}{2}}} + \frac{1+(-1)^n}{4}   \cdot \lambda_{\floor{\frac{n+1}{2}}+1} \geq 0
\end{align*}
on the eigenvalues of the K\"ahler curvature operator yields
\begin{align*}
g(  \Ric(R), \overline{R})  \geq 0.
\end{align*}

Hence the maximum principle shows that $R$ is parallel. Moreover, if the inequality is strict, then $g(  \Ric(R), \overline{R}) > 0$ unless $|R{^{\mathfrak{u}}}|^2 = 0.$ According to lemma \ref{HatOnKaehlerCurvatureOperators}, this is the case if and only if $R$ has constant holomorphic sectional curvature. \hfill $\Box$ 

\begin{example} \normalfont 
In the proof of Theorem \ref{KaehlerTachibana} we used that for every $L \in \mathfrak{u}(n)$ and every $R \in \operatorname{Sym}^2(\mathfrak{u}(n))$ we have
\begin{align*}
|LR|^2 \leq 8 |L|^2 | \mathring{R}|^2.
\end{align*}
This estimate is optimal.

Following the notation of example \ref{EigenvalueConditionIndependentForOrthBisectCurv},  define an algebraic curvature operator $\mathfrak{R}$ by setting $\mu_{1,+}=3,$ $\mu_{2,+}=\mu_{3,+}=0$ and $\mu_{1,-}=-1$, $\mu_{2,-}=1,$ $\mu_{3,-}=3.$ Note that $\mathfrak{R}$ is Einstein. Let $R \in \operatorname{Sym}_B^2(\mathfrak{u}(2))$ denote the associated K\"ahler curvature tensor. 

Since $| g \left( (\Xi_{i,\pm}) \Xi_{j,\pm}, \Xi_{k,\pm} \right)| = \sqrt{2}$ if $\lbrace i, j, k \rbrace = \lbrace 1, 2, 3 \rbrace$ and all signs agree, and zero otherwise, example \ref{HatsOfCurvatureTensors} 
\label{OptimalityEstimateKaehlerCurvatureOperators} and lemma \ref{HatOnKaehlerCurvatureOperators} imply that $| \mathring{R}|^2 = \frac{| R^{\mathfrak{u}}|^2}{12} = 8$ and $|\Xi_{1,+}R|^2=0,$ $|\Xi_{1,-}R|^2=2 | \mathring{R}|^2 ,$ $|\Xi_{2,-}R|^2=8 | \mathring{R}|^2,$ $|\Xi_{3,-}R|^2=2 | \mathring{R}|^2.$ In particular, $| \Xi_{2,-}R|^2$ achieves equality in the above estimate.
\end{example}




\begin{thebibliography}{Flat}
\bibitem[{Ale}68]{AlekseevskiiRiemannianSpacesExceptionalHol}
{Alekseevskii, D.V.}, \emph{Riemannian spaces with exceptional holonomy
  groups}, Funct. Anal. Appl. \textbf{2} (1968), 97--105.

\bibitem[BG65]{BishopGoldbergCohomKaehler}
R.~L. Bishop and S.~I. Goldberg, \emph{On the second cohomology group of a
  {K}aehler manifold of positive curvature}, Proc. Amer. Math. Soc. \textbf{16}
  (1965), 119--122.

\bibitem[Boc46]{BochnerVectorFieldsAndRic}
S.~Bochner, \emph{Vector fields and {R}icci curvature}, Bull. Amer. Math. Soc.
  \textbf{52} (1946), 776--797.

\bibitem[Boc49]{BochnerCurvatureAndBettiNumbersII}
\bysame, \emph{Curvature and {B}etti numbers. {II}}, Ann. of Math. (2)
  \textbf{50} (1949), 77--93.

\bibitem[Bre10]{BrendleEinsteinNIC}
Simon Brendle, \emph{Einstein manifolds with nonnegative isotropic curvature
  are locally symmetric}, Duke Math. J. \textbf{151} (2010), no.~1, 1--21.

\bibitem[Che57]{ChernGeneralizationKaehlerGeometry}
Shiing-shen Chern, \emph{On a generalization of {K}\"{a}hler geometry},
  Algebraic geometry and topology. {A} symposium in honor of {S}. {L}efschetz,
  Princeton University Press, Princeton, N. J., 1957, pp.~103--121.

\bibitem[Che07]{ChenPosOrthBisectCurv}
X.~X. Chen, \emph{On {K}\"{a}hler manifolds with positive orthogonal
  bisectional curvature}, Adv. Math. \textbf{215} (2007), no.~2, 427--445.

\bibitem[CST09]{ChenSunTianKaehlerRicciSolitons}
Xiuxiong Chen, Song Sun, and Gang Tian, \emph{A note on {K}\"{a}hler-{R}icci
  soliton}, Int. Math. Res. Not. IMRN (2009), no.~17, 3328--3336.

\bibitem[Fuj87]{FujikiCohomKaehlerSymplectic}
Akira Fujiki, \emph{On the de {R}ham cohomology group of a compact {K}\"{a}hler
  symplectic manifold}, Algebraic geometry, {S}endai, 1985, Adv. Stud. Pure
  Math., vol.~10, North-Holland, Amsterdam, 1987, pp.~105--165.

\bibitem[Gal81]{GallotSobolevEstimates}
Sylvestre Gallot, \emph{Estim\'{e}es de {S}obolev quantitatives sur les
  vari\'{e}t\'{e}s riemanniennes et applications}, C. R. Acad. Sci. Paris
  S\'{e}r. I Math. \textbf{292} (1981), no.~6, 375--377.

\bibitem[GM75]{GallotMeyerCurvOperatorAndForms}
S.~Gallot and D.~Meyer, \emph{Op\'{e}rateur de courbure et laplacien des formes
  diff\'{e}rentielles d'une vari\'{e}t\'{e} riemannienne}, J. Math. Pures Appl.
  (9) \textbf{54} (1975), no.~3, 259--284.

\bibitem[Gu09]{GuNewProofGeneralizedFrankelConj}
Hui-Ling Gu, \emph{A new proof of {M}ok's generalized {F}rankel conjecture
  theorem}, Proc. Amer. Math. Soc. \textbf{137} (2009), no.~3, 1063--1068.

\bibitem[GW72]{GreeneWuCurvatureAndComplexAnalysis}
R.~E. Greene and H.~Wu, \emph{Curvature and complex analysis. {II}}, Bull.
  Amer. Math. Soc. \textbf{78} (1972), 866--870.

\bibitem[GZ10]{GuZhangExtensionMokTheorem}
HuiLing Gu and ZhuHong Zhang, \emph{An extension of {M}ok's theorem on the
  generalized {F}rankel conjecture}, Sci. China Math. \textbf{53} (2010),
  no.~5, 1253--1264.

\bibitem[HSW81]{HowardSmythWuNonnegBiSect}
Alan Howard, Brian Smyth, and H.~Wu, \emph{On compact {K}\"{a}hler manifolds of
  nonnegative bisectional curvature. {I}}, Acta Math. \textbf{147} (1981),
  no.~1-2, 51--56.

\bibitem[KW70]{KobayashiWuHolomorphicSectionsOfHermitianVB}
Shoshichi Kobayashi and Hung-Hsi Wu, \emph{On holomorphic sections of certain
  hermitian vector bundles}, Math. Ann. \textbf{189} (1970), 1--4.

\bibitem[Li80]{LiSobolevConstant}
Peter Li, \emph{On the {S}obolev constant and the {$p$}-spectrum of a compact
  {R}iemannian manifold}, Ann. Sci. \'{E}cole Norm. Sup. (4) \textbf{13}
  (1980), no.~4, 451--468.

\bibitem[Mok88]{MokUniformizationKaehler}
Ngaiming Mok, \emph{The uniformization theorem for compact {K}\"{a}hler
  manifolds of nonnegative holomorphic bisectional curvature}, J. Differential
  Geom. \textbf{27} (1988), no.~2, 179--214.

\bibitem[Mor79]{MoriProjectiveManifoldsWithAmpleTangentBundles}
Shigefumi Mori, \emph{Projective manifolds with ample tangent bundles}, Ann. of
  Math. (2) \textbf{110} (1979), no.~3, 593--606.

\bibitem[Mui60]{MuirTreatiseOnDeterminants}
Thomas Muir, \emph{A treatise on the theory of determinants}, Revised and
  enlarged by William H. Metzler, Dover Publications, Inc., New York, 1960.

\bibitem[MW93]{MicallefWangNIC}
Mario~J. Micallef and McKenzie~Y. Wang, \emph{Metrics with nonnegative
  isotropic curvature}, Duke Math. J. \textbf{72} (1993), no.~3, 649--672.

\bibitem[NZ18]{NiZhengComparisonAndVanishingKaehler}
Lei Ni and Fangyang Zheng, \emph{Comparison and vanishing theorems for
  {K}\"{a}hler manifolds}, Calc. Var. Partial Differential Equations
  \textbf{57} (2018), no.~6, Paper No. 151, 31.

\bibitem[NZ20]{NiZhengPositivityAndKodaira}
\bysame, \emph{{Positivity and Kodaira embedding theorem}}, arXiv:1804.09696
  (2020).

\bibitem[Pet16]{PetersenRiemGeom}
Peter Petersen, \emph{{Riemannian Geometry}}, third ed., Graduate Texts in
  Mathematics, vol. 171, Springer, 2016.

\bibitem[Poo80]{PoorHolonomyProofPosCurvOperatorThm}
W.~A. Poor, \emph{A holonomy proof of the positive curvature operator theorem},
  Proc. Amer. Math. Soc. \textbf{79} (1980), no.~3, 454--456.

\bibitem[PW20]{PetersenWinkNewCurvatureConditionsBochner}
Peter Petersen and Matthias Wink, \emph{{New Curvature Conditions for the
  Bochner Technique}}, Invent. Math. (2020).

\bibitem[SY80]{SiuYauCompactKaehlerPosBiholCurv}
Yum~Tong Siu and Shing~Tung Yau, \emph{Compact {K}\"{a}hler manifolds of
  positive bisectional curvature}, Invent. Math. \textbf{59} (1980), no.~2,
  189--204.

\bibitem[Tac74]{TachibanaPosCurvOperator}
Shun-ichi Tachibana, \emph{A theorem on {R}iemannian manifolds of positive
  curvature operator}, Proc. Japan Acad. \textbf{50} (1974), 301--302.

\bibitem[Wey39]{WeylClassicalGroups}
Hermann Weyl, \emph{The {C}lassical {G}roups. {T}heir {I}nvariants and
  {R}epresentations}, Princeton University Press, Princeton, N.J., 1939.

\bibitem[Wil13]{WilkingALieAlgebriaicApproach}
Burkhard Wilking, \emph{A {L}ie algebraic approach to {R}icci flow invariant
  curvature conditions and {H}arnack inequalities}, J. Reine Angew. Math.
  \textbf{679} (2013), 223--247.

\bibitem[Wu81]{WuCompactKaehlerNonnegBiSectII}
H.~Wu, \emph{On compact {K}\"{a}hler manifolds of nonnegative bisectional
  curvature. {II}}, Acta Math. \textbf{147} (1981), no.~1-2, 57--70.

\bibitem[Yan18]{YangRCpositivity}
Xiaokui Yang, \emph{R{C}-positivity, rational connectedness and {Y}au's
  conjecture}, Camb. J. Math. \textbf{6} (2018), no.~2, 183--212.

\bibitem[Yau82]{YauProblems}
Shing~Tung Yau, \emph{Problem section}, Seminar on {D}ifferential {G}eometry,
  Ann. of Math. Stud., vol. 102, Princeton Univ. Press, Princeton, N.J., 1982,
  pp.~669--706.

\end{thebibliography}

\end{document}